\documentclass[a4paper, 11pt]{article}
\usepackage{amsfonts,latexsym,graphicx,amssymb,amsmath,url,amsthm,color}
\usepackage{enumerate}
\usepackage[english]{babel}
\usepackage{url}
\usepackage[colorlinks,linkcolor=blue,anchorcolor=blue,citecolor=blue]{hyperref}
\newtheorem{theorem}{Theorem}[section]
\newtheorem{sketch-theorem}{sketch-Theorem}
\newtheorem{proposition}[theorem]{Proposition}
\newtheorem{corollary}[theorem]{Corollary}
\newtheorem{lemma}[theorem]{Lemma}
\newtheorem{conjecture}[theorem]{Conjecture}

\theoremstyle{remark}

\newtheorem{remark-number}[theorem]{Remark}
\newtheorem*{remark}{{\bf{Remark}}}

\newcommand{\Z}{{\mathbb{Z}}}
\newcommand{\F}{{\mathbb{F}}}
\newcommand{\E}{{\mathbb{E}}}
\newcommand{\Prob}{{\mathbb{P}}}

\begin{document}

\title{Caps and progression-free sets in $\Z_m^n$}
\author{Christian Elsholtz, P\'{e}ter P\'{a}l Pach}

\date{\today}

\maketitle

\begin{abstract}
We study progression-free sets in the abelian groups $G=(\Z_m^n,+)$.
Let $r_k(\Z_m^n)$ denote the maximal size of
a set $S \subset \Z_m^n$ that does not contain a proper arithmetic
progression of length $k$. We give lower bound constructions, which e.g.~include that 
$r_3(\Z_m^n) \geq C_m \frac{((m+2)/2)^n}{\sqrt{n}}$, when $m$ is even. When $m=4$ this is of order at least $3^n/\sqrt{n}\gg \vert G \vert^{0.7924}$.
Moreover, if the progression-free set
$S\subset \Z_4^n$ satisfies a technical 
condition, which dominates the problem at least in low dimension, then $|S|\leq 3^n$ holds.

We present a number of new methods which cover lower bounds
for several infinite families of parameters $m,k,n$, which includes for example: $r_6(\Z_{125}^n) \geq (85-o(1))^n$.

For $r_3(\Z_4^n)$ we determine the exact values, when $n \leq 5$,
e.g.~$r_3(\Z_4^5)=124$,
and for $r_4(\Z_4^n)$ we determine the exact values, when $n \leq 4$,
e.g.~$r_4(\Z_4^4)=128$.
\end{abstract}

\section{Introduction}

There has been great interest in finding progression-free sets in $\Z_m^n:=(\Z/(m \Z))^n$, especially when $m=3$ or $4$. 
When $m=3,4,5$ the properties
``no arithmetic progression of length $3$ modulo $m$'' and ``no $3$ points on  any line'' are equivalent. The last property is also well known under the name cap-sets.
In spite of this great interest in progression-free sets and caps there is not much literature on progression-free sets in 
$\Z_m^n$, in the case of general $m>3$, and of general progressions of length $k$, and hardly any explicit values of the maximal size of such sets is known.\footnote{There is certainly an extensive literature in the related area
of finite geometry over finite fields, but in literature from an additive combinatorics point of view we are essentially aware of an exercise in the book by Tao and Vu, and a paper by Lin and Wolf, details below.}

This paper intends to fill this gap and provides several new techniques to find lower bounds, and even to find exact values in the case $m=4$, which are comparable to the known values for $m=3$.

However, before we come to this, we briefly summarize a number of related questions.
The problem of finding sets $S \subset \Z_m^n$ with, or without, 
a given property has been investigated frequently. Often one is actually 
interested in the maximal size of $|S|$. Also, often even the one-dimensional
case has been of fundamental interest.
Let us recall some of the properties that have been investigated.
\begin{enumerate}[\hspace{8pt}]
\item[1)]
 Erd\H{o}s and Tur\'{a}n \cite{ErdosandTuran:1936} raised the problem
of studying the maximal size $r_k(N)$
of sets in $\{1, \ldots, N\}$ without an arithmetic progression of 
length $k$. There are important contributions by
Behrend, Bloom and Sisask, Bourgain, Gowers, Green, Roth, Salem and Spencer, Sanders, 
Szemer\'{e}di, Tao
\cite{Behrend:1946, Bourgain:1999, Gowers:2001, Green-Tao:2009, Roth:1953, SalemandSpencer:1942,
Sanders:2011, Szemeredi:1975}. In particular, the proof of $r_k(N)=o(N)$, 
as $N$ tends to infinity, and quantitative versions thereof, proved to be very
influential in this area.
It is interesting to note that
the size of progression-free sets even enters the complexity of matrix
multiplication, see \cite{CoppersmithandWinograd:1990, Williams:2012}.

The question of arithmetic progressions has also been studied modulo $m$, see e.g.~Croot \cite{Croot:2008}.
In this setting ``modulo $m$'' one has to clarify if elements of the progression
can occur more than once. For example $(1,3,1,3)$ can possibly be considered 
as a progression of length $4$ modulo $m=4$. In this paper, however, 
we study ``proper arithmetic progressions''
meaning that all elements in the progression are {\emph{distinct}}, unless otherwise stated.

\item[2)]
Assume that $S$ does not have $k$ elements $x_1, \ldots , x_k\in \Z_m^n$
that satisfy (for fixed constants $a_1, \ldots , a_k \in \Z$)
a linear equation 
\[a_1 x_1 + a_2x_2 + \cdots + a_kx_k=0 \in \Z_m^n.\]
\begin{enumerate}
\item
The case $n=1, k=m, a_1=a_2 = \cdots =a_k=1$ was first investigated by 
Erd\H{o}s, Ginzburg and Ziv \cite{Erdos-Ginzburg-ZiV:1961}, 
who proved that for any $2m-1$ elements in 
$\Z_m$, where in this problem repetition is allowed, 
there exists a subset of $m$ elements with sum 
$0 \in\Z_m$. (There are hundreds of papers on generalizations and variants,
the general topic is called ``zero sums in finite abelian groups'').
In the case $n=2$ there has been important work by Reiher \cite{Reiher:2007}.
The multidimensional case with $n \geq 3$ is widely open,
even though there are lower bounds by
Edel, Elsholtz et al. \cite{Elsholtz:2004, 
Edel-Elsholtz:2007, Edel:2008}, and upper bounds
by Alon and Dubiner \cite{AlonandDubiner:1995}, 
Naslund \cite{Naslund} and Heged\"{u}s \cite{Hegedus}.
\item The case $x_1+x_2-x_3=0, x_i \in S$ corresponds to sum-free sets. In the one-dimensional
case $S \subset \{1, \ldots , m\}$ it
is known that the maximal size is 
$|S| \leq \lfloor \frac{m}{2}\rfloor +1$, if all $x_i$ are distinct, or
$|S| \leq \lfloor \frac{m+1}{2}\rfloor$ if $x_1=x_2$ is allowed.
In the case modulo a prime $m$ it follows from the Cauchy-Davenport theorem
that the maximal size satisfies $|S|\leq \frac{m+1}{3}$ ($x_i$ all distinct).

In the multidimensional case of an integer grid there are results by 
Cameron \cite{Cameron}, Elsholtz and Rackham \cite{Elsholtz-Rackham:2017}.
\end{enumerate}
\item[3)] The case of no geometric line (of $m$ points)
in the integer grid $\{1, \ldots , m\}^n$ is known as Moser's cube problem,
see \cite{Moser:1970, Polymath}. Closely related is the question of finding the maximal number of lattice points in the same cube
$\{1, \ldots , m\}^n$, but without any combinatorial line.
The famous theorem of Hales-Jewett \cite{Hales-Jewett} of $o(m^n)$ points, when $m$ is fixed and $n$ tends to infinity, became very influential.

\end{enumerate}
In this paper we concentrate on sets $S\subset \Z_m^n$ of maximal size $|S|=r_k(\Z_m^n)$
with no $k\leq m$ elements in arithmetic progression. 
Observe that an arithmetic progression of length $k$ 
can be expressed by means of $k-2$ linked linear conditions $x_i-2x_{i+1}+x_{i+2}=0, i=1,
\ldots, k-2$.

The multidimensional case of no 3 points in 
arithmetic progression has frequently 
been studied, especially modulo $m=3$. 
Here the questions of ``no zero sums $x_1+x_2+x_3=0$''
and ``no arithmetic progression $x_1+x_3=2x_2$'' turn out to be equivalent as
$1\equiv -2 \bmod 3$.
The problem is known as the ``cap set problem''.
There were important contributions by Brown and Buhler \cite{BrownandBuhler:1982},
Frankl, Graham and R\"{o}dl \cite{Frankl-Graham-Rodl:1987},
Meshulam \cite{Meshulam:1995}, Lev \cite{Lev:2004}, 
Bateman and Katz \cite{BatemanandKatz:2012},
Croot, Lev and Pach \cite{Croot-Lev-Pach},
Ellenberg and Gijswijt \cite{EllenbergandGijswijt:2016}.

For a long time it was 
an important open problem if there is a $\delta>0$ such that
 $r_k(\Z_m^n) <(3-\delta)^n$ holds, for all progression-free sets $S\subset \Z_3^n$.
Various authors mentioned this statement with varying degree of certainty or doubt,
(see Alon and Dubiner \cite{AlonandDubiner:1993}, \cite{AlonandDubiner:1995},
Green \cite{Green:2005}, Kalai \cite{Kalai:blog},
Edel \cite{Edel:2004}, Tao \cite{Tao:2007})
until the solution by
Croot, Lev and Pach \cite{Croot-Lev-Pach} (when $m=4$), and finally
Ellenberg and Gijswijt \cite{EllenbergandGijswijt:2016}.


Meshulam's \cite{Meshulam:1995} long-standing bound $r_3(\Z_m^n)=O(\frac{m^n}{n})$ for odd values of $m\geq 3$ was extended by Lev \cite{Lev:2004} to even values $m\geq 4$.
Improving this,
Sanders \cite{Sanders:2009}
proved the following result:
\[ r_3(\Z_4^n) = O\left( \frac{4^n}{n \log^{c} n} \right),\]
for some positive $c$. Green and Tao \cite{Green-Tao:2009} write that $c=2^{-22}$ is admissible.
Introducing an entirely new approach, based on the polynomial method rather than Fourier techniques,
Croot, Lev and Pach \cite{Croot-Lev-Pach} proved
that
\[ r_3(\Z_4^n)\leq 4^{\gamma n}= 3.61\ldots^n,\]
where $\gamma \approx 0.926$.
The new methods introduced in  \cite{Croot-Lev-Pach}
also led to the result
in the case $r_3(\Z_3^n)\leq 2.756^n$ by Ellenberg and Gijswijt
\cite{EllenbergandGijswijt:2016}.
Again, the case of cap sets has applications to the complexity of matrix multiplication, see
\cite{Alon-Shpilka-Umans:2013, Blasiak-Church-etal}.

The corresponding problem on lower bounds of progression-free sets 
in $G=(\Z_3^n,+)$ has also been studied in detail. It is known (see \cite{Edel:2004} for the history and current record)
that there is a set $S$ with 
$\vert S \vert > 2.217389^n= \vert G \vert^{\beta}$ with 
$\beta= \frac{\log 2.217389}{\log 3}\approx 0.724851)$. The currently 
strongest lower bound example comes from a
product construction, based on an example in dimension 480. 
 
For a lower bound when $m=4$ Sanders \cite{Sanders:2009} proved:
there exists $S\subset G=(\Z_4^n,+)$ which does not
contain a proper three term
arithmetic progression with 
\[ \vert S  \vert \gg \vert G \vert^{2/3}\approx 2.519^n.\]
This result follows  from finding an example in $\Z_4^3$  with $16$ 
elements and using a product construction. (Note that $\sqrt[3]{16}=2.519\ldots$.)

The following is known:
\[ 2.2174\ldots ^n \leq r_3(\Z_3^n) \leq 2.756\ldots ^n, 
\cite{Edel:2004,EllenbergandGijswijt:2016}\]
\[ 2.519\ldots^n \ll r_3(\Z_4^n) \leq 3.62\ldots^n, \cite{Sanders:2009,Croot-Lev-Pach},\]
and for primes $p\geq 3$ and some positive constant $\delta_p$
\[  r_3(\Z_p^n) \leq (p-\delta_p)^n, \cite{EllenbergandGijswijt:2016}.\]
\begin{remark-number}
From this one can conclude that for every $m \geq 3$ there 
exists a positive $\delta_m$ such that
\begin{equation}\label{consequence-for-all-m}
    r_3(\Z_m^n) \leq (m-\delta_m)^n. 
\end{equation}
For powers of $2$ this follows from \cite{Croot-Lev-Pach}, all other values have an odd prime factor such that it follows from
\cite{EllenbergandGijswijt:2016}.
\end{remark-number}
There are only very few explicit values known. 
In the case of cap sets modulo $m=3$ the following is known:
\[r_3(\Z_3^1)=2, r_3(\Z_3^2)=4,
r_3(\Z_3^3)=9, r_3(\Z_3^4)=20,r_3(\Z_3^5)=45, r_3(\Z_3^6)=112.\]
The author of the 6-dimensional result (Potechin \cite{Potechin:2008}), and the authors of the {\em classification} of the unique
5-dimensional maximum cap \cite{Edel-Ferret:2002},
 (required for the 6-dimensional case by Potechin) mentioned they used computer calculations. We would like to thank Y.~Edel for informing us that for the  paper
\cite{Edel-Ferret:2002} the computation time was a few weeks.

The remaining part of the paper is organized as follows:
After some necessary notation and describing the results we first prove the asymptotic lower bounds in Section~\ref{secproof}, as these proofs are shorter. In Section~\ref{sec:reformulation} we give a reformulation for the problem of finding $r_3(\mathbb{Z}_4^n)$ and $r_4(\mathbb{Z}_4^n)$. In Section~\ref{sec:const} we give a lower bound for $r_3(\mathbb{Z}_4^n)$,
we then prove that this construction gives the exact values
up to dimension 5 (Sections~\ref{sec3ap1} and \ref{sec3ap2}), which require some detailed case studies. Finally, in Section~\ref{sec4ap} we prove the exact values for $r_4(\mathbb{Z}_4^n)$ up to dimension 4.

\section{Notation}

We use the Landau $O$ and $o$-notation such as $f(n)=O_t(g(n))$, where the $O$-constant depends at most on a parameter $t$. We also use the Vinogradov notation, where $f(n)\ll_t g(n)$ or $g(n)\gg_t f(n)$ has the same meaning as the $O$-expression above.

In Sections \ref{sec3ap1}, \ref{sec3ap2}, \ref{sec4ap} we will work with linear and affine subspaces of $\mathbb{F}_2^n$. If $L$ is a linear subspace of dimension $d$, for brevity we will say that $L$ is a $d$-subspace. The smallest linear  subspace containing the vectors $v_1,\dots,v_k$ will be denoted by $\langle v_1,\dots,v_k\rangle$. 

Similarly, if $L$ is an affine subspace of dimension $d$, we will say that $L$ is an affine $d$-subspace and the smallest affine subspace containing $v_1,\dots,v_k$ will be denoted by $\langle v_1,\dots,v_k\rangle_{aff}$.


Throughout the paper for a subset $A\subseteq\mathbb{F}_2^n$ we use the notation $A+A=\{a+a':a,a'\in A\}$ for the sumset and $A\hat{+}A=\{a+a':a,a'\in A,a\ne a'\}$ for the restricted sumset.

\section{Results and methods}
\begin{theorem}\label{3APthm}
For sets without arithmetic progression of length $3$ 
we have the following results:
\[r_3(\Z_4^1)=2, r_3(\Z_4^2)=6, r_3(\Z_4^3)=16, r_3(\Z_4^4)=42,
r_3(\Z_4^5)=124.\]
\end{theorem}
We give quite uniform proofs for all these dimensions.
The value $r_3(\Z_4^3)=16$ was stated before by Sanders
\cite{Sanders:2009} (and was indeed a computer calculation by O.~Sisask), 
and the value $r_3(\Z_4^4)=42$ was determined in a Masters' Thesis by
Lawrence Newcombe \cite{Newcombe:2008} (a student of the first author). From that proof it was already 
apparent that $r_3(\Z_4^n)$ could be much smaller than $4^n$, due to a $\Z_2^n$-substructure of $\Z_4^n$, but proceeding to 
higher dimension might have been very tedious.

Next we give a lower bound on $r_3(\mathbb{Z}_4^n)$. In the construction we use binary codes with certain minimum distances. Let $C(m,d)$ denote the largest possible size of a (possibly non-linear) code in $\mathbb{F}_2^m$ with minimum distance at least $d$.
(In the literature, this is often denoted with $A(m,d)$.)
Note that $C(m,1)=2^m$ (all vectors can be taken) and $C(m,2)=2^{m-1}$ (all codewords can be taken with even Hamming-weight).
Here are links to tables of exact values of maximal codes or bounds:
\url{https://www.win.tue.nl/~aeb/codes/binary-1.html}
and
\url{http://www.codetables.de/}

\begin{theorem}\label{const} For $n>1$ we have $r_3(\mathbb{Z}_4^n)\geq \max\limits_{0\leq t\leq n}\sum\limits_{i=t+1}^n \binom{n}{i}C(i,i-t)$.
\end{theorem}
As a consequence of this result one can prove a quite good lower bound.
\begin{corollary}{\label{result1}}
\[ r_3(\Z_4^n) \gg \frac{3^n}{\sqrt{n}}\]
which implies that there exists a progression-free set 
$S\subset \Z_4^n$ with
\[ \vert S \vert \gg   4^{0.7924n}.\]
\end{corollary}

As this Corollary is the first nontrivial case of the lower bound
constructions and is suitable for discussing various methods we give two proofs of it.
The first one is a direct application of Theorem \ref{const}.


\begin{proof}[Proof of Corollary (Proof 1)]
Calculations show that the optimal choice for $t$ in Theorem~\ref{const} satisfies $t\sim 2n/3$. In particular, 
for $2\leq n\leq 10$ the optimal choice is $t=\lceil (2n-5)/3\rceil$.
Note that the sum of only the first two terms in the lower bound $\sum\limits_{i=t+1}^n \binom{n}{i}C(i,i-t)$, with an optimal value of $t$, is
$$\binom{n}{t+1}2^{t+1}+\binom{n}{t+2}2^{t+1}\sim 1.5 \binom{n}{2n/3}\sim \frac{9}{4\sqrt{\pi}}\cdot\frac{3^n}{\sqrt{n}}.$$
The total sum is not much larger as it is bounded above by $\frac{3}{\sqrt{\pi}}\cdot\frac{3^n}{\sqrt{n}}$
(see also \cite{Chvatal:1972}).
\end{proof}

Here we give a brief sketch of the second proof, full details and more motivation is in
 section \ref{secproof}.
We consider all elements in $\Z_4^n$ with exactly $\lfloor n/3\rfloor $ entries being $1$, and the remaining entries being $0$ or $2$. This gives $\binom{n}{\lfloor n/3\rfloor} 2^{n-\lfloor n/3\rfloor }\gg \frac{3^n}{\sqrt{n}}$ elements. The reason why this set is progression-free is that there is a unique middle-point of a putative progression and that the boundary points cannot use their 1-entries in a nontrivial way, such that
a progression pattern 012 cannot occur. But then there cannot be progressions of three {\emph{distinct}} points, a contradiction.

Finally, we found a third proof,
based on weighted Sperner capacity of the 2-vertex graph with 
one directed edge, and vertex weights 1 and 2, but decided not 
to include it.

\begin{corollary}\label{gen_lower}
$$2\leq r_3(\mathbb{Z}_4^1),\quad 6\leq r_3(\mathbb{Z}_4^2),\quad 16\leq r_3(\mathbb{Z}_4^3),\quad 42\leq r_3(\mathbb{Z}_4^4), \quad 124\leq r_3(\mathbb{Z}_4^5),$$
$$344\leq r_3(\mathbb{Z}_4^6),\quad 960\leq r_3(\mathbb{Z}_4^7),\quad 2832\leq r_3(\mathbb{Z}_4^8),\quad 7880\leq r_3(\mathbb{Z}_4^9), \quad 22232\leq r_3(\mathbb{Z}_4^{10}).$$
\end{corollary}
Let us explain this with two examples:
when $n=5$, choose $t=2$. Then 
\[\begin{array}{rcl}
r_3(\Z_4^5)& \geq &\binom{5}{3}C(3,1)+\binom{5}{4}C(4,2)+\binom{5}{5}C(5,3)\\
&=&10 \cdot 8+5 \cdot 4+1 \cdot 4=80+40+4=124,
\end{array}\]
which is best possible by Theorem \ref{3APthm}.
When $n=8$, choose $t=4$.
\[\begin{array}{rcl}
r_3(\Z_4^8) &\geq &\binom{8}{5}C(5,1)+\binom{8}{6}C(6,2)+\binom{8}{7}C(7,3)
+ \binom{8}{8}C(8,4)\\
&=&56 \cdot 32+28\cdot 32 +8 \cdot 16+ 1 \cdot 16=2832. 
\end{array}\]
\begin{theorem}{\label{thm:r_4(Z_4^n}}
For sets without arithmetic progression of length $4$ 
we have the following results:
\[r_4(\Z_4^1)=3, r_4(\Z_4^2)=10, r_4(\Z_4^3)=36, r_4(\Z_4^4)=128.\]
\end{theorem}

It is well known that results of this type can be lifted to 
higher dimensions and yield asymptotic results 
by a simple product construction, compare also Proposition 3.5
\cite{Edel-Elsholtz:2007} in the similar setting of zero-sum free sets.
\begin{lemma}{\label{product-construction}}
\begin{itemize}
\item[a)] Let $S_1 \subset \Z_m^{n_1}$ and $S_2 \subset \Z_m^{n_2}$
be $k$-progression-free sets, then 
$S_1 \times S_2 \subset \Z_m^{n_1+n_2}$ is also $k$-progression-free.
\[r_k(\Z_m^{n_1+n_2})\geq r_k(\Z_m^{n_1})\, r_k(\Z_m^{n_2}).\]

\item[b)] A repeated application of part a) gives:
\[r_k(\Z_m^{nt})\geq \left(r_k(\Z_m^n)\right)^t.\]
\end{itemize}
\end{lemma}
Lifting the largest known exact values $r_3(\Z_4^5)=124$ and $r_4(\Z_4^n)$ gives:
\begin{corollary}
\[ r_3(\Z_4^n)\gg 2.622^n,
\quad  r_4(\Z_4^n)\gg 3.363^n.\]
\end{corollary}
The first result is considerably weaker than Corollary {\ref{result1}}, while the second one is the strongest that is currently known.
The product construction  only makes use of ``local'' information from small dimensions.
The ``relative density'' for the high dimensional problem 
is the same as 
for the low dimensional base-example that was lifted.
Lifting for example the bound $r_3(\Z_{4}^{10})\geq 22232$ gives
a better estimate $r_3(\Z_4^n)\gg 2.720\ldots^n$.
For $k=3$ it is better to use the ``global'' information 
from the digits giving the lower bound $\frac{3^n}{\sqrt{n}}$. But for $k=4$ we do not know how to replace the product construction by a better strategy.

In many cases we present constructions much better than the product construction. These make use of
``global'' properties i.e.~making full use of the actual dimension $n$.
With our current understanding this only works when $k<m$. For $k=m$ the product 
construction appears to be the strongest available method, see also Edel \cite{Edel:2004}.

These proofs describe a set explicitly in terms of its coordinate entries, similar to
the constructions by Salem and Spencer \cite{SalemandSpencer:1942}, 
and Behrend \cite{Behrend:1946}.
Salem and Spencer constructed progression-free sets in the integers by 
representing integers in an $m$-ary digit system, $m$ odd,
and using the digits $0 \leq a_i \leq (m-1)/2$ a fixed number of times, namely with frequency $n/d$ for integers of length $n$. Restricting the digits avoids wrapping over modulo $m$. Behrend constructed large progression-free sets in the integers by mapping a high-dimensional sphere, which by convexity is progression-free, to the integers. He also 
represented integers in an $m$-ary system with digits $0 \leq a_i \leq (m-1)/2$, where $m$ is odd, and fixed value $\sum_{i=1}^n a_i^2$.
In the integer case the optimization of the values of $m$ and $n$ shows that Behrend's construction is greatly superior. In our setting we make use of both ideas, and observe that $m,n$ are fixed by the problem, and the method of Behrend, when applicable, is only slightly stronger, but a bit more complicated.


\begin{proposition}{\label{limit}}
Let $k\geq 3$ and $m\geq 3$ be fixed.
The limit 
\[\alpha_{k,m}:= \lim_{n \rightarrow \infty} \left(r_k(\Z_m^n)\right)^{1/n}\]
exists.
\end{proposition}
 It follows from Theorems  
 {\ref{r_3(Z_m^n)-odd-m}, \ref{thm:m-even}}
  that
$\lceil\frac{m+1}{2}\rceil\leq \alpha_{k,m} \leq m$, and also
$\alpha_{k,m}<m$, when $k=3$.

In view of the above results, and also in view of an upper bound in a relevant case, see Theorem \ref{subspace_thm},
we state the following conjecture:
\begin{conjecture}{\label{conj_3-o(1)}}
\[ r_3(\Z_4^n)=(3-o(1))^n, \text{ i.e. } \alpha_{3,4}=3\]
\end{conjecture}





Tao and Vu \cite[exercise 10.1.3]{TaoandVu:2006}
observe that there is a construction in $\Z_m^n$ with
at least $\frac{[m/2]^n}{m^2 n^2}$ points without 3-progression
(based on Behrend's construction).\footnote{It seems they possibly intended the denominator to be $mn^2$ (in our notation).}

Lin and Wolf \cite{Lin-Wolf:2010} proved the following:
If $m$ is a prime and $k\leq m$
\[r_k(\Z_m^n) \geq \left(m^{2(k-1)}+m^{k-1}-1\right)^{\frac{n}{2k}}
\approx m^{\frac{(k-1)n}{k}}.\]
Their proof makes use of a product construction, as explained in Lemma \ref{product-construction}.
They also have some results, when $m$ is a pure prime power, but this refers to finite fields $\F_m$, which are different from $\Z_m$.
In particular, when $m$ is prime and $m^{k-1}$ is large, and $n$ increases, the exponential growth of the lower bound is based on the constant $m^{\frac{k-1}{k}}$, compared to 
$\lfloor \frac{m+2}{2}\rfloor$ here.


We now give our general theorems, which improve the above lower bound and remove the prime condition on $m$:
\begin{theorem}{\label{r_3(Z_m^n)-odd-m}}
Let $m\geq 5$ be odd. There exists some $C_m>0$ such that
\[r_3(\Z_m^n) \geq \frac{C_m}{\sqrt{n}}
\left(\frac{m+1}{2}\right)^n.\]
Moreover,
with $\sigma_m=\sqrt{\frac{1}{2880}\left(m^4+4m^3-14m^2-36m+45\right)}$
the value $C_m=\frac{1}{3\sqrt{3}\, \sigma_m}$ is admissible.
For increasing odd $m$ asymptotically $C_m \sim \frac{8 \sqrt{5}}{\sqrt{3}\, m^2}$ holds.
\end{theorem}
In the case $m=3$ this would give a lower bound of $\gg\frac{2^n}{\sqrt{n}}$ only which is smaller than the trivial lower bound by taking all $2^n$ elements with coordinate entries $0$ or $1$.
Also note that in view of $r_k(\Z_m^n) \geq r_3(\Z_m^n)$ the Theorem trivially 
induces lower bounds for any $k\geq 3$ (also in the theorem below).

A crucial idea again is to avoid any product construction and to use
one more digit than Tao and Vu \cite[exercise 10.1.3]{TaoandVu:2006}  used, with some extra constraints, which are less costly (if $m$ is constant and $n$ increases). 
Their lower bound $\frac{m^n}{2^n}\cdot\frac{1}{ m^2 n^2}$ in case $m=4$
would also be weaker than the trivial progression-free set $\{0,1\}^n$ with $2^n$ elements.
\begin{theorem}{\label{thm:m-even}}
Let $m\geq 4$ be even.
There exists some $C_m>0$ such that
\[ r_3(\Z_m^n) \geq 
\frac{C_m}{\sqrt{n}}\left(\frac{m+2}{2}\right)^n.\]
With
 $\sigma_m = \sqrt{\frac{m^4+8m^3+4m^2-48m}{2880}}$ one can choose
$C_m=\frac{1}{3\sqrt{3}\sigma_m}$.
For large $m$ one has that $C_m \sim \frac{8 \sqrt{5}}{\sqrt{3}\, m^2}$.
\end{theorem}
(A version of this result, in the special case $m=8$ has also been observed in \cite{Petrov-and-Pohoato:2018}, 
having seen a precursor of this manuscript. Their main concern is an improvement of the upper bound.)

As is well known from Behrend's construction there are good reasons to restrict to half of the available digits. In the above cases we go up to one element more than half of the digits. In the cases of even $m$ one additionally has to study progressions of type $0 \frac{m}{2}0$ carefully. In the examples below we go even further, and note that those progressions which actually use the reduction modulo $m$ cause quite a bit of extra work. 
(For example, in the case $r_4(\Z_{11}^n)$ we have to care about
progressions of type $1,6,0,5$ modulo $11$.)

\begin{theorem}{\label{k=4,mod11}}
The following holds
\[r_4(\Z_{11}^n) \gg \frac{7^n}{n^3}. \]
\end{theorem}
(No attempt was made to reduce the exponent $3$.)
For comparison Lin and Wolf \cite{Lin-Wolf:2010} have a lower bound of about about $6.04^n$. (For fixed $k$ the improvement increases, as $m$ increases.)

It is clear that on a case by case study one can prove related results for several individual values of $m$ and $k$. Here we present two further cases where these ideas are generalized to   
infinite families $m=p^s, k=p^{s-1}+1$ (or $k=p^{s-2}+1$ respectively), 
where $p$ is prime.
It should be noted that in this case the set of digits used is not
consecutive, but makes use of the structure of
orbits of length $p$, and hence the algebraic structure. As can be seen, several good properties are preserved: many progression types can be excluded by the Salem-Spencer ``same-frequency property", and the ``all-elements-distinct" property, (i.e.~proper progressions).
\begin{theorem}{\label{thm:m-primepower-progression}}
Let $m=p^s$ be a pure prime power, $s\geq 2$.
Let $k=p^{s-1}+1$. Then there exist constants $C_m>0$ and $0< c_m \leq m/2$ such that the following holds: 
\[r_k(\Z_m^n)\geq C_m \frac{(m-p+1)^n}{n^{c_m}}.\]
\end{theorem}
\begin{corollary}
There exist positive constants $C_m$ and $c_m\leq m/2$ such that the following holds:
\[r_3(\Z_4^n)\geq C_m \frac{3^n}{n^{c_m}}.\]
\[r_5(\Z_8^n)\geq C_m \frac{7^n}{n^{c_m}},\]
\[r_{10}(\Z_{27}^n)\geq C_m \frac{25^n}{n^{c_m}}.\]
\[r_{26}(\Z_{125}^n)\geq C_m \frac{121^n}{n^{c_m}}.\]  
\[r_{102}(\Z_{101^2}^n)\geq C_m \frac{10101^n}{n^{c_m}}.\]
\end{corollary}

\begin{theorem}{\label{thm:primepowerprogressions-k=p^{s-2}+1}}
Let $m=p^s$ be a pure prime power, $s\geq 3$.
Let $k=p^{s-2}+1$. 
Then there exist constants $C_m>0$ and $0< c_m \leq m/2$ such that the following
holds: 
\[r_k(\Z_m^n)\geq C_m \frac{(m-2p^2+2p)^n}{n^{c_m}}.\]
\end{theorem}
For $p=2$, this is certainly not best possible. By Theorem
\ref{thm:m-even} for $m=8, k=3$ one can use
5 digits, rather than 4.
\begin{corollary}
There exist positive constants $C_m$ and $c_m\leq m/2$ such that the following holds:
\[r_{p+1}(\Z_{p^3}^n)\geq C_m \frac{(p^3-2p^2+2p)^n}{n^{c_m}}.\]
\[r_{4}(\Z_{27}^n)\geq C_m \frac{15^n}{n^{c_m}}.\]
\[r_{82}(\Z_{729}^n)\geq C_m \frac{717^n}{n^{c_m}}.\]
\[r_6(\Z_{125}^n)\geq C_m \frac{85^n}{n^{c_m}}\]
\[r_{26}(\Z_{625}^n)\geq C_m\frac{585^n}{n^{c_m}}.\]
\end{corollary}
We are not aware of any earlier results of this type.



\ \\
We now briefly discuss some aspects of the proofs of the exact values, and of a conditional upper bound.
For the estimations of $r_3(\mathbb{Z}_4^n)$ we shall need a reformulation of the problem which is presented in Section {\ref{sec:reformulation}}. Let us say that a system of subsets $A(x)\subseteq \mathbb{F}_2^n$ ($x\in\mathbb{F}_2^n$) satisfies property $(*)$, if 
the following implication holds:
$$
\forall x\in \mathbb{F}_2^n  \ (y\in x+A(x)\hat{+}A(x) \implies A(y)=\emptyset). \eqno{(*)}
$$
(Note that for $A(x)=\emptyset$ we define $x+A(x)\hat{+}A(x):=\emptyset$.)
In Lemma~\ref{3apeq} we will show that the answer to this question is exactly $r_3(\mathbb{Z}_4^n)$, that is, estimating the maximal total size of a system of subsets $\{A(x):x\in\mathbb{F}_2^n\}$ satisfying $(*)$ is equivalent with our original question.

As it turns out it is very useful 
that we can reduce the
case of arbitrary subsets to the case of subspaces.
We do not know, if this can be done for higher dimension, but for 
the low dimensions studied here explicitly this
is a quite powerful method.
In this case, the upper bound $O(3^n)$ is quite close to the general lower bound in the unrestricted case,
namely $r_3(\Z_4^n)\gg 3^n/\sqrt{n}$.
This is the heuristic reason why we state conjecture \ref{conj_3-o(1)}.

\begin{theorem}\label{subspace_thm}
If the system of subsets $A(x)$ satisfies $(*)$ and all non-empty subsets $A(x)$ are subspaces, then $\sum\limits_{x\in \mathbb{F}_2^n} |A(x)|\leq  3^n$.
\end{theorem}

Note that for $n=1$ any $2$-element subset forms a progression-free subset in $\Z_4^n$. If $n\in\{2,3,4\}$, then the extremal construction is also unique in the following sense:

\begin{theorem}\label{unique}
Let $n\in \{2,3,4\}$. 
If the systems of subsets $\{A(x):x\in \mathbb{F}_2^n\}$ and $\{A'(x):x\in \mathbb{F}_2^n\}$ both have total size $r_3(\mathbb{Z}_4^n)$ and they satisfy $(*)$, then there is an invertible affine linear transformation $\varphi: \mathbb{Z}_2^n\to \mathbb{Z}_2^n$ and vectors $c(x)\in \mathbb{Z}_2^n$ ($x\in \mathbb{Z}_2^n$) such that $A'(x)=A(\varphi(x))+c(x)$ for every $x\in \mathbb{Z}_2^n$.
\end{theorem}

\section{Proofs of the asymptotic lower bounds}\label{secproof}
We will use several times that the central multinomial coefficients
can be approximated by Stirling's formula:
\begin{lemma}{\label{multinomial}}
Let $d \geq 2$ be an integer. There exists a constant $c_d$ such that
\[ \binom{ d n}{n,\ldots , n}\sim c_d \frac{d^{dn}}{n^{(d-1)/2}}.\]
\end{lemma}
Here we give a geometrically inspired proof of
Corollary \ref{result1}, which is independent of Theorem 
\ref{const}.

\begin{proof}[Proof of Corollary \ref{result1} (Proof 2):]
The crucial idea is that an arithmetic progression of length 3 
(with 3 distinct points) in $\Z_4^n$ 
has a uniquely defined middle point. For comparison, 
this is not the case in $\Z_3^n$.

We relate the problem to a problem posed by Leo Moser \cite{Moser:1970}.
Find in $H=\{0,1,2\}^n$ the maximal set of elements without 
``three on a line''.
(which is also known as Moser's cube problem). Observe that in this case 
there is no reduction modulo 3.
Let $f(n)$ denote the largest such number in $H=\{0,1,2\}^n$.
It is known that $f(1)=2, f(2)=6, f(3)=16$,
(see Chv\'{a}tal \cite{Chvatal:1973}), $f(4)=43$ (see Chandra \cite{Chandra:1973}),
$f(5)=124, f(6)=353$ \cite{Polymath}.
In dimensions 1, 2, 3 and  5 these values are the same as $r_3(\Z_4^n)$, but in
dimension 4 one has that $r_3(\Z_4^4)=42<f(4)=43$.

A simple observation by Koml\'{o}s \cite{Komlos} shows that
$f(n) \gg \frac{3^n}{\sqrt{n}}$, and the implicit constant
 was refined again by Chv\'{a}tal \cite{Chvatal:1972}. The construction by Chv\'{a}tal
relates the problem to coding theory and gives 
$f(5) \geq 124$, for example.

Let us adapt Koml\'{o}s' \cite{Komlos} observation to our situation:
the set 
\[ S= \{(x_1, \ldots , x_n)\in \{0,1,2\}^n: 
x_i=1 \text { for } m=\lfloor n/3 \rfloor 
\text{ values } i\}\]
has the claimed number of elements and has no three points on a 
line.

Let us count the number of such points, let 
$n$ be  a multiple of $3$, then by Stirling's formula 
$S$ has 
\[
\begin{array}{ccl}
\vert S\vert&=&
2^{n-m}\binom{n}{m}
= 2^{2n/3}\binom{n}{n/3}\\
&\sim &
2^{2n/3}\frac{\sqrt{2 \pi n}n^n}{e^n}
\frac{e^{n/3}}{\sqrt{2\pi n/3}(n/3)^{n/3}}
\frac{e^{2n/3}}{\sqrt{2\pi 2n/3}(2n/3)^{2n/3}}
\gg \frac{3^n}{\sqrt{n}}
\end{array} \]
elements. When $n\equiv 1,2 \bmod 3$ 
we have the same order of magnitude, up to a constant factor, for example,
by filling the extra 1 or 2 coordinates with entries from $\{0,1\}$. 
Further observe that for three points $P_1, P_2, P_3$ to be on a line (in this
order), one would need, in each coordinate, that \\
i) all entries are the same,\\
or ii) the entries are $0,1,2$ or $2,1,0$ (in this order). 
Since the number of ``middle entries 1'' is constant for all points, there
cannot be an arithmetic progression of three distinct digits.

Let us embed the set $S$ from $\{0,1,2\}^n$ canonically into
$G=(\Z_4^n,+)$. Think of  $G$ as the lattice points $\{0,1,2,3\}^n$ 
but now with reduction modulo $4$ in each coordinate.
Observe that the set $S$ does not have a single ``3''-entry.
An arithmetic progression of length 3
modulo $4$ that does not make use of $x_i=3$ in any
coordinate must be of one of the types below,
in a given coordinate.\\
The digits are:\\
i) the same,\\
ii) or are $0,1,2$ or $2,1,0$ in this order,\\ 
iii) or $0,2,0$, or $2,0,2$.

We will show that the set $S\subset \Z_4^n$ does not contain a proper 
3-progression. Suppose $S$ does contain three distinct points $P_1,P_2,P_3$
in arithmetic progression.
The case i) where all entries are the same does not play any role.
Let us look at those coordinates where the entries differ.
Since all points have the same number of 1 entries, let us study,
where one of the three elements uses a ``1'', but another point does not:
For this, the only possibilities are  $0,1,2$ and $2,1,0$.
But here only the middle point $P_2$ 
can make use of a 1. So, the two points $P_1$ and $P_3$ 
cannot make use of their ones, unless all three
entries are identically 1. This means that all three points have their 
ones in exactly the same position, and that there is no coordinate with a progression 
012 or 210. So, let us look at the other coordinates.
The only possibilities left are 020 or 202.
But then $P_1$  and $P_3$  would be the very same point, a
contradiction to the definition of a proper progression.
\end{proof}
\begin{proof}[Proof of Proposition \ref{limit}]
The idea of this proof might go back to Shannon \cite{Shannon:1956}, see also
Davis and Maclagan \cite{Davis-Maclagan}.
Let $\alpha_{k,m}(n)=\left(r_k(\Z_m^n)\right)^{1/n}$, so that we have the following properties:
By the product construction  (Lemma \ref{product-construction}) we have 
\[ r_k(\Z_m^{n_1})r_k(\Z_m^{n_2}) \leq  r_k(\Z_m^{n_1+n_2}),\] 
i.e.
$\alpha_{k,m}(n_1)^{n_1}\alpha_{k,m}(n_2)^{n_2}\leq \alpha_{k,m}(n_1+n_2)^{n_1+n_2}$ and therefore 
\[n_1\log \alpha_{k,m}(n_1)+n_2 \log \alpha_{k,m}(n_2) \leq (n_1+n_2) \log \alpha_{k,m}(n_1+n_2).
\]
Therefore, the sequence $\{n \log \alpha_{k,m}(n)\}_{n=1}^{\infty}$ is superadditive.
By Fekete's Lemma on superadditive sequences
the limit $\lim_{n \rightarrow \infty} \log \alpha_{k,m}(n)$ exists and equals $\sup_n \log \alpha_{k,m}(n)$.

By Theorems \ref{r_3(Z_m^n)-odd-m} and \ref{thm:m-even} (proofs below)
we know that 
$m^n\geq r_k(\Z_m^n)\geq r_3(\Z_m^n)\gg_m 
\lceil\frac{m+1}{2}\rceil^n \frac{1}{n^{c_m}}$ holds. Hence, for each $k\geq 3$
 we have $\lceil \frac{m+1}{2}\rceil \leq \alpha_{k,m} \leq m$. When $k=3$ 
it follows for all $m \geq 3$ that $\alpha_3 <m$, for example by applying
Ellenberg and Gijswijt \cite{EllenbergandGijswijt:2016} to any odd prime divisor of $m$, and Croot, Lev and Pach \cite{Croot-Lev-Pach} otherwise.

\end{proof}

\begin{proof}[Proof of Theorem {\ref{r_3(Z_m^n)-odd-m}}:]
We first prove a slightly weaker result
based on the Salem-Spencer construction \cite{SalemandSpencer:1942}
for sets of integers without arithmetic 3-progression.
Recall that $m$ is odd and that we only need to study
$k=3$. Assume first that $n$ is a multiple of $(m+1)/2$.
Choose vectors with digits
\[a_i \in \left\{0,1,2,\ldots , \frac{m-1}{2}\right\}\]
with exactly 
$n_i$ entries of digit $i$, where $i\in \left\{0,1,2,\ldots , \frac{m-1}{2}\right\}$.
The number of such vectors is maximized when $n_i= \frac{n}{\frac{m+1}{2}}$ for every $i$.
This gives at least $C_m (\frac{m+1}{2})^n\frac{1}{n^{c_m}}$ points, for positive constants $C_m,c_m$.
If $n$ is not a
multiple of $(m+1)/2$ one can fill the remaining coordinates with entries
$0 \leq a_i <k$,
which slightly weakens the constant $C_m$.

We show that there is no arithmetic $3$-progression:
by the choice of the allowed digits, if the digit $a> 0$ occurs, 
then the digit $m-a\equiv -a \bmod m$ is forbidden, so $0$ is never in the centre of a proper 3-progression. 
As all vectors have the same number of 0-entries, all of these
 digits 0 must occur in the same coordinate position, giving a trivial $000$-progression. 
One then continues:
All nontrivial 3-progressions, without the digit 0 do not have a digit 1 in the centre, 
and hence the digit $1$ can only come from a $111$-progression.


To do an explicit example, let  $m=11, k=3$, we use the 
digits: $0,1,2,3,4,5$. A complete list of all possible 3-progressions of these digits is:
\[  \begin{cases}
000,111,222,333,444,555\\
012, 024, 123, 135, 234, 210, 345, 321, 420, 432, 531, 543.
\end{cases}.\]
As there are three distinct points, there must be  a proper 3-progression
of 3 distinct digits $abc$.
As the digit $0$ is never in the centre of any of these nontrivial 3-progressions,
and as all vectors have the same number of $0$-entries,
the digit can only occur in the trivial way: $000$. This leaves the following 
shorter list of nontrivial 3-progressions:
\[123, 135, 234,  321,  345, 432, 531, 543.\]
Now the digit $1$ is never in the centre, and 1 can only occur in the trivial $111$ progression.
leaving the list $234, 345, 432, 543$.
Now, the digit 2 is never in the centre, so 2 can only occur as $222$, leaving
$345, 543$.
Now $3$ is never in the centre, which gives the final contradiction.

Note that initially we have restricted the frequency of all digits
$0,1,2,3,4,5$, 
but we can now observe that restricting the frequency of the digits $0,1,2,3$ is enough.

We now prove the theorem in its full strength, based on Behrend's construction. The number of elements used is larger by a factor $n^c$ only.

Let $m$ be odd, and $n$ be a multiple of $(m+1)/2$.
Let 
\[
S_R=\left\{ ( a_1, \ldots, a_n): a_i \in \{0,1, \ldots , (m-1)/2\}, \sum_{i=1}^n \left(a_i- \frac{m-1}{4}\right)^2=R\right\}.\]
Here $S_R$ can be thought of as a sphere 
about centre $\left( (m-1)/4, \ldots, (m-1)/4)\right)$
with
$R$ as squared radius.
We prove that all $S_R$ are progression-free and
there exists an $S_R$ of size at least
$C_m \frac{1}{\sqrt{n}}\left(\frac{m+1}{2}\right)^n$.

Suppose there are three distinct points $P_1, P_2, P_3$ in arithmetic progression.
None of the progressions in a fixed coordinate
makes use of the reduction modulo $m$, so that convexity of the geometric sphere gives a contradiction. But let us look at this arithmetically:
Let the progression in the $i$-th coordinate be $a_i-d_i, a_i, a_i+d_i$.
Then for the three points one has that
$\sum_{i=1}^n (a_i-d_i-\frac{m-1}{4})^2=\sum_{i=1}^n (a_i-\frac{m-1}{4})^2=
\sum_{i=1}^n (a_i+d_i-\frac{m-1}{4})^2$.
Then
\[ \sum_{i=1}^n \left( \left(a_i+d_i-\frac{m-1}{4}\right)^2 + \left(a_i-d_i-\frac{m-1}{4}\right)^2- 2 \left(a_i-\frac{m-1}{4}\right)^2\right) =0.\]
This gives $\sum_{i=1}^n 2d_i^2$=0. Hence $d_i=0$ for all $i$. In other words, the three points are identical, which is a contradiction.
The size of large sets $S_R$ follows from the observation that 
most elements in  $(a_1, \ldots , a_n)\in [0, \frac{m-1}{2}]^n$
have a value of $R=\sum_{i=1}^n (a_i- \frac{m-1}{4})^2$ in an interval
of size the standard deviation around the mean value. 
To make this more precise, we follow Elkin \cite{Elkin:2011} 
and consider $a_i- \frac{m-1}{4}$ as independent random variables $Y_1, \ldots, Y_n$, distributed uniformly in $\{-(m-1)/4, \ldots, (m-1)/4\}$, and 
$Z_i=Y_i^2, Z=\sum_{i=1}^n Z_i, i \in \{1, \ldots, n\}$.
The expected value is
$\mu_m :=\E(Z_i)=\frac{1}{(m+1)/2}\sum_{i=-(m-1)/4}^{(m-1)/4}i^2=
\frac{1}{48}m^2+\frac{1}{24}m-\frac{1}{16}$ and
$\E(Z)=n \E(Z_i)$.
The variance is
\[\begin{array}{rcl}
Var(Z_i)&=&\E(Z_i^2)-\E(Z_i)^2\\
&=&
\frac{m^4}{1280}+\frac{m^3}{320}
-\frac{11m^2}{1920}-\frac{17m}{960}+\frac{5}{256}
-\left( \frac{1}{48}m^2+\frac{1}{24}m-\frac{1}{16} \right)^2\\
&=&
\frac{1}{2880}\left(m^4+4m^3-14m^2-36m+45\right),
\end{array}
\]
and $Var(Z)=n Var(Z_i)$.
The standard deviation is 
$\sigma_m=\sqrt{Var(Z_i)}$ and 
$\sigma_Z=\sqrt{Var(Z)}=\sigma_m \sqrt{n}$, 
where $\sigma_m$ depends only on $m$. By Chebychev's inequality $\Prob(|Z-\E(Z)|>a\sigma_Z)\leq \frac{1}{a^2}$. With $a=\sqrt{3}$ we see that for at least two thirds of all elements in $[0, \frac{m-1}{2}]^n$ the sum of digit squares-distances from the centre point
$\left(\frac{m-1}{4}, \ldots, \frac{m-1}{4}\right)$ is in the interval 
$[\mu_m n -a\sigma_Z, \mu_m n+a\sigma_Z]$.
By the pigeonhole principle there exists a squared radius $R$  
with frequency at least
$\frac{C_m}{\sqrt{n}}\left( \frac{m+1}{2}\right)^n$, where 
$C_m=\frac{2}{3\cdot 2 \sqrt{3}\sigma_m}= \frac{1}{3\sqrt{3}\sigma_m}$.

Note that $\sigma_5=\frac{\sqrt{2}}{3}, \sigma_7=1, \sigma_9=\sqrt{\frac{14}{5}}$.
As the proof only makes use of effective bounds, the result is valid for all odd $m\geq 5$ and all $n$.
If the odd value $m$ tends to infinity, then,  
asymptotically $\sigma_m \sim \frac{m^2}{24 \sqrt{5}}$ holds, 
giving the claimed value of $C_m$.

\end{proof}
\begin{remark}
While the Salem-Spencer type construction with all frequencies of the digits being constant is completely explicit, the above Behrend-type proof uses the pigeonhole principle, which is not explicit, and in algorithmic terms slowly, as one would need to search for a good value $R$.
However, a result of Rankin \cite{Rankin:1960} gives entirely explicit bounds on the number of representations of numbers as a sum of $n$ squares of bounded size. In particular this shows that not only there are good values $R$ but that {\emph{all}}
values $R$ in the interval are good, when weakening the constant $C_m$ by a small factor only. In particular, one can choose $R= \lfloor \mu n \rfloor$.
In another direction, as the above argument does not make use of reduction modulo $m$, it seems possible to implement the improvement by Elkin \cite{Elkin:2011}, 
which might gain extra factor, maybe of size $n^c$. Elkin observed that
3-progressions in a suitable union of spheres (annulus) are geometrically quite restricted. One can then prove that there is a large subset of this union which is progression-free. 
\end{remark}

\begin{proof}[Proof of Theorem {\ref{thm:m-even}}:]
Again, we first prove a sightly weaker version based on the Salem-Spencer construction.
This proof is similar to the previous case, but as $m$ is even there is one extra complication to care for.
Assume first that $n$ is a multiple of $(m+2)/2$, and that there is an arithmetic 
progression of three distinct points.

Choose vectors with exactly
$n_i$ entries of digit $i$, where $i\in \{0,1,2,\ldots , \frac{m}{2}\}$.
The number of such vectors is maximized when  $n_i=\frac{n}{\frac{m+2}{2}} $ for every $i$.
This gives at least $(\frac{m+2}{2})^n\frac{C'_m}{n^{c_m}}$ points. 
If $n$ is not a
multiple of $(m+2)/2$ one can fill the remaining coordinates with 0-entries,
which will slightly weaken the constant $C'_m$.

Working out the set of all nontrivial 3-progressions, one observes that
the boundary values $0$ and $m/2$ occur as values in the middle position only in the progressions
of type $0 \frac{m}{2} 0$, $\frac{m}{2} 0\frac{m}{2}$ or constant progressions.
This means that the values of $0$ or $\frac{m}{2}$ can occur in constant 3-progressions,
$000$, $\frac{m}{2} \frac{m}{2}  \frac{m}{2}$ and the same number of
progressions of type $0 \frac{m}{2} 0$ and $\frac{m}{2} 0\frac{m}{2}$.
Hence other nontrivial progressions using $0$, or $\frac{m}{2}$,  like $012$ never occur.

By definition of a 3-progression we search for three {\em distinct} points, this means there must 
be somewhere another nontrivial progression $abc$ with three distinct digits 
in $\{1,2, \ldots , \frac{m}{2}-1\}$. One can then continue iteratively as before, 
and concludes there is no nontrivial 3-progression of 3 distinct points. 
Let us define the Behrend-sphere:
\[S_R=\{(a_1, \ldots, a_n): a_i \in \{0,1, \ldots , m/2\}, \sum_{i=1}^n \left(a_i- \frac{m}{4}\right)^2=R\}.\] 
We prove that $S_R$ is 3-progression-free in $\Z_m^n$. The estimate on the number of points is as in the case of odd $m$ above.

Suppose there are three distinct points $P_1,P_2,P_3$ in arithmetic progression.
The non-constant progressions in a fixed coordinate
do not make use of the reduction modulo $m$, with the two exceptions of $0\frac{m}{2}0$ and $\frac{m}{2}0\frac{m}{2}$.
Let $n_1, n_2, \ldots, n_s$ denote the number of coordinates with a fixed progression-pattern such as $000$, $012$, $024$ etc.
Of these, let $n_1$ count the pattern $0 \frac{m}{2}0$ and
and let $n_2$ count the pattern $\frac{m}{2}0\frac{m}{2}$. As all other patterns do not wrap over modulo $m$ let
$n_i$ count the pattern $p_i-d_i, p_i, p_i+d_i$.

Hence $\sum_{i=1}^s n_i=n$.
The points $(a_1, \ldots, a_n)$ in $S_R$ lie on a sphere with centre $(m/4, \ldots, m/4)$.
Let the progression pattern of the $j$-th coordinates be $p_j-d_j, p_j, p_j+d_j$.
Then for the three points $P_1,P_2,P_3$ one has that
$n_1\frac{m^2}{16} +n_2\frac{m^2}{16} +\sum_i n_i (p_i-d_i-\frac{m}{4})^2=
n_1\frac{m^2}{16} +n_2\frac{m^2}{16} +
\sum_i n_i (p_i-\frac{m}{4})^2=
n_1\frac{m^2}{16} +n_2\frac{m^2}{16} +
\sum_i n_i (p_i+d_i-\frac{m}{4})^2$.
Then
\[ \sum_{i=1}^s n_i \left((p_i+d_i-\frac{m}{4})^2 + (p_i-d_i-\frac{m}{4})^2- 2 (p_i-\frac{m}{4})^2 \right) =0.\]
This gives $\sum_i n_i 2d_i^2=0$. Hence for all patterns with $d_i\neq 0$ one has that $n_i=0$.
The three points only consist of patterns $aaa$, $0\frac{m}{2}0$ or $\frac{m}{2}0\frac{m}{2}$.
Therefore the first and the third point are exactly the same point, in contradiction to the assumption.

We estimate $C_m$ as above:
for $i=1, \ldots , n$ consider $Y_i=a_i- \frac{m}{4}$ as independent random variables, distributed uniformly in $\{-m/4, \ldots, m/4\}$, and 
$Z_i=Y_i^2, Z=\sum_{i=1}^n Z_i, i \in \{1, \ldots, n\}$.
The expected value is
$\mu_m :=\E(Z_i)=\frac{1}{(m+2)/2}\sum_{i=-m/4}^{m/4}i^2=
\frac{1}{48}m^2+\frac{1}{12}m$ and
$\E(Z)=n \E(Z_i)$.
The variance is
\[\begin{array}{rcl}
Var(Z_i)&=&\E(Z_i^2)-\E(Z_i)^2\\
&=&
\frac{m^4}{1280}+\frac{m^3}{160}
+\frac{m^2}{120}-\frac{m}{60}
-\left( \frac{1}{48}m^2+\frac{1}{12}m \right)^2\\
&=&
\frac{1}{2880}\left(m^4+8m^3+4m^2-48m\right),
\end{array}
\]
and $Var(Z)=n Var(Z_i)$.
The standard deviation is 
$\sigma_m=\sqrt{Var(Z_i)}$ and 
$\sigma_Z=\sqrt{Var(Z)}=\sigma_m \sqrt{n}$, 
where $\sigma_m$ depends only on $m$. By Chebychev's inequality $\Prob(|Z-\E(Z)|>a\sigma_Z)\leq \frac{1}{a^2}$. With $a=\sqrt{3}$ we see that for at least two thirds of all elements in $[0, \frac{m}{2}]^n$ the sum of digit squares-distances from the centre point
$\left(\frac{m}{4}, \ldots, \frac{m}{4}\right)$ is in the interval 
$[\mu_m n -a\sigma_Z, \mu_m n+a\sigma_Z]$.
By the pigeonhole principle there exists a squared radius $R$  
with frequency at least
$\frac{C_m}{\sqrt{n}}\left( \frac{m+2}{2}\right)^n$, where 
$C_m=\frac{2}{3\cdot 2 \sqrt{3}\sigma_m}= \frac{1}{3\sqrt{3}\sigma_m}$.

Note that $\sigma_4=\frac{\sqrt{2}}{3}, \sigma_6=1, \sigma_8=\sqrt{\frac{14}{5}}$.
As the proof only makes use of effective bounds, the result is valid for all even $m\geq 4$ and all $n$.
If the even value of $m$ tends to infinity, then  
asymptotically $\sigma_m \sim \frac{m^2}{24 \sqrt{5}}$ holds, 
giving the claimed value of $C_m$.

Note that the values of the constants in the two cases $m$ odd and even are quite similar.

\end{proof}

\begin{proof}[Proof of Theorem {\ref{k=4,mod11}}]
Let $n$ be a multiple of $7$ and let $D=\{0,1,2,3,4,5,6\}$.
Let 
\[S=\{(a_1, \ldots , a_n):a_i \in D \text{ and for each $j\in D$  there are $n/7$ values $i \in \{1, \ldots, n\}$ with $a_i=j$} \}.\]
The list of trivial and nontrivial arithmetic progressions of length $4$
with digits in $D$ modulo $11$ is:
\[  \begin{cases}
0000,1111,2222,3333,4444,5555,6666\\
0\underline{1}23, 0246, 123\underline{4}, 16\underline{0}5, 
\underline{2}345, 3456, 
32\underline{1}0, 
\underline{4}321, 5\underline{0}61, 543\underline{2}, 6420, 6543
\end{cases}
\]
Let $d(a_1 a_2 a_3 a_4)$ denote the number of coordinates, where the pattern
$a_1a_2a_3a_4$ occurs among the 4 points which are in arithmetic progression.
As the digit $0$ occurs in all 4 positions with the same frequency, and applying it to positions 3 and 1 we see that
the number of occurrences of a pattern $1605$ equals the sum of the number of occurrences of patterns $0123$ and $0246$ together. (See underlined symbols in the list of patterns.)
Also looking at digit 1 at positions 2 and 1,
and combining these gives: 
\[\begin{array}{lll}
(1)&d(1605)&=d(0123)+d(0246)\\
(2)&d(0123)&=d(1605)+d(1234)=d(0123)+d(0246)+d(1234), \\
&&\text{ which implies:}\\
(3)&d(1234)&=0.
\end{array}\]
As $1234$ is the {\emph{only}} nontrivial progression with digit 4 in the last position,
all 4's must occur in form of a trivial progression, 4444.
Therefore 
\[d(0246)=d(1234)=d(2345)=d(3456)=d(4321)=d(5432)=d(6420)=d(6543)=0.\]
This leaves only the following nontrivial progressions.
\[
0123,  1605,  3210,  5061\]
Here we observe that there are no digits 2 or 6 at the boundary, and also no digits 3 or 5 in the positions 2 and 3.
So, in each coordinate there can only be a constant progression, which contradicts that we have a proper progression of distinct points in $S$.
The number of elements in $S$ is the multinomial coefficient 
$\binom{n}{n/7,n/7,n/7,n/7,n/7,n/7,n/7}=\frac{n!}{\left((n/7)!\right)^7}\sim C \frac{7^n}{n^3}$ for some constant $C>0$, by Stirling's formula.
If $n$ is not a multiple of $7$, say $n=7r+i$, one adds $i\leq 6$ further coordinates with constant digits, which weakens the overall lower bound by a small factor.
\end{proof}

\begin{proof}[Proof of Theorem {\ref{thm:m-primepower-progression}}]

In this situation we do not take the digits consecutively, but make use of the algebraic structure
of $(\Z_m,+)$. In particular $p^{s-1}$ generates a subgroup of order $p$, and $k$-progressions in $\Z_m$ with gap size divisible by $p$ have the property that the first element is the same as the last element. We choose the digits as follows:
\[ D= \Z_m \setminus \{i p^{s-1}-1: i=2, \ldots, p\}.\]

Observe that $D$ contains $p^{s-1}-1$ complete cycles of length $p$, and one extra element, and so $|D|=(p^{s-1}-1)p+1=p^s-p+1$.
There are three types of progressions of length $k=p^{s-1}+1$ in $D$:
\begin{enumerate}
    \item
Type I progressions have a non-zero gap size divisible by $p$. In this
case the first element and the last element of the progression are the   
same.
\item For Type II progressions the gap size is not divisible by $p$.
In this case all residue classes modulo $p^{s-1}$ occur, and the first
and last element are the same modulo $p^{s-1}$, but
cannot be the same modulo $m=p^s$.
 The residue class $p^{s-1}-1 \bmod m$ must occur, as $D$ contains only
one element $-1 \bmod p^{s-1}$.
We observe that no such $k$-progression can start with $p^{s-1}-1$, as it
would have to end at {\emph{another}}
element $-1 \bmod p^{s-1}$, which is impossible.
\item
Type III progressions are constant.
\end{enumerate}
So far this was the part which generalized the algebraic situation from
$m=4$ to prime powers. The last part is the set-theoretic trick inspired
by Salem and Spencer.

Let $|D|\mid n$ and let
\[S=\left\{ (a_1, \ldots, a_n):a_i \in D, \forall d\in D:
\left| \{j\in [1,n]: a_j=d\}\right|  =\frac{n}{|D|} \right\}.\]
The number $|S|$ of elements is the multinomial coefficient
$\binom{n}{n/|D|, \ldots ,n/|D|}\sim C_m \frac{|D|^n}{n^{(|D|-1)/2}}$ according to Lemma~\ref{multinomial}.
Suppose that $S$ contains a {\emph{proper}} arithmetic progression of length $k=p^{s-1}+1$.

Let us study the occurrence of the digit $p^{s-1}-1$ in the first vector.
It cannot be part of a type I or type II progression, and hence must be a
constant type III progression.
Therefore all coordinate entries  $p^{s-1}-1$ in all vectors occur in the
same positions. In all other coordinates we only have type I and type III
progressions. For these the first and the last elements are the same,
modulo $m$. Hence there cannot be a {\emph{proper}} arithmetic progression
of length $k$, which by definition consists of $k$ distinct elements. 

\end{proof}

\begin{proof}[Proof of Theorem \ref{thm:primepowerprogressions-k=p^{s-2}+1}]
Recall that $m=p^s, s\geq 3, k=p^{s-2}+1$.
Let $D_1=\left\{p^{s-2}i:i =0, \ldots, p^2-1\right\}$,
\[D_2=\left\{p^{s-1}i+j:i\in \{0, \ldots, p-1 \}, j \in \{1, \ldots, p-1 \} \right\} \]
and $D_3=\{0,1,2,,\ldots , p-1\}$.
Choose the digits:
\[ D= \left( \Z_m \setminus (D_1\cup D_2)\right) \cup D_3.\]
Observe that $|D|=p^s-p^2-p(p-1)+p=p^s-2p^2+2p$.
For example, when $m=27,k=4$ , then 
\[D=\{0,1,2, 4,5,7,8,13,14,16,17,22,23,25,26\}.\]
There are four types of progressions of length $k=p^{s-2}+1$ in $D$:\\
Type I progressions with gap size $p^t, 2\leq t< s$, which therefore 
contain a cycle of length $p^{s-t}$. Here the first element and the last element is the same. Note that the class $0$ cannot be part of such a progression, as the element $p^t\cdot  p^{s-1-t}=p^{s-1}$ is not in $D$.\\
Type II: progressions of gap size $p$. They must use exactly one of
the digits in $D_3\setminus \{0\}$, but cannot use it in the first or last position:
starting with $d\in D_3\setminus \{0\}$ and gap size $p$ 
the longest progression size is $k-1$, as otherwise a digit in $D_2$ would be needed, which is impossible. Also the progression cannot contain $0$, as it would then also contain $p\cdot p^{s-2}$, which is impossible.
(Example, $m=27$: the  longest progression with gap size $3$ is:
$22,25,1,4,7$.)\\
Type III progressions have a gap size coprime to $p$, and do not contain any cycle. They consist of $k=p^{s-2}+1$ distinct digits, and in particular go through all residue classes modulo $p^{s-2}$, and therefore contain the special element $0$.
But note that no such progression can start with $0$, as it would also have to end at \emph{another}
element $0\bmod p^{s-2}$, which is impossible.\\
Type IV progressions are constant.

Note that progressions starting with $0$ must be of type IV.
Now let $|D|$ divide $n$ and let 
\[S=\left\{ (a_1, \ldots, a_n):a_i \in D, \forall d\in D: 
\left| \{j\in [1,n]: a_j=d\}\right| =\frac{n}{|D|} \right\}.\]
The number $|S|$ of elements is the multinomial coefficient
$\binom{n}{n/|D|, \ldots ,n/|D|}\sim C_m \frac{|D|^n}{n^{(|D|-1)/2}}$.
As all elements contain the same number of $0$-entries, the constant progressions (type IV) are the only ones that contain any $0$-entry.

Now suppose that $S$ has a proper progression of length $k=p^{s-2}+1$.
All $k$ elements contain in $\frac{n}{|D|}$ positions an entry $d\in D_3$.
Looking at the first element of the progression 
we see that these progressions 
starting with $d\in D_3$ can only be of type IV, i.e.~constant.
Hence all digits $D_3$ cannot take part in any nontrivial progression.
With all other digits in 
$\Z_m \setminus (D_1\cup D_2)$
and with all progression types
we observe that the first and the last elements are the same. 
Altogether, the set $S$ of vectors  
does not have a {\emph{proper}} arithmetic progression of length $k$, 
which by definition consists of $k$ distinct elements.
\end{proof}


\section{Subset reformulation}{\label{sec:reformulation}}

In this section we give a ``subset formulation'' for the question of determining $r_3(\mathbb{Z}_4^n)$ and $r_4(\mathbb{Z}_4^n)$. As an application of the former one, we give another proof for Theorem~\ref{result1}, then we prove Theorem~\ref{subspace_thm}.

\subsection{Reformulation for 3AP-free-ness}
Let us say that a system of subsets $A(x)\subseteq \mathbb{F}_2^n$ ($x\in\mathbb{F}_2^n$) satisfies property $(*)$, if 
the following implication holds:
$$
\forall x\in \mathbb{F}_2^n  \ (y\in x+A(x)\hat{+}A(x) \implies A(y)=\emptyset). \eqno{(*)}
$$
(Note that for $A(x)=\emptyset$ we define $x+A(x)\hat{+}A(x):=\emptyset$.)
Let $r_3'(n)$ denote the maximal possible size of $\sum \limits_{x\in\mathbb{F}_2^n} |A(x)|$, if the system of subsets $\{A(x):x\in \mathbb{F}_2^n\}$ satisfies $(*)$.

The proof of Lemma~\ref{3apeq} (below) shows that property $(*)$ nicely captures the condition that the ``corresponding'' $A\subseteq \mathbb{Z}_4^n$ is 3AP-free.

\begin{lemma}\label{3apeq}
For every $n\geq 1$ we have $r_3(\mathbb{Z}_4^n)=r_3'(n)$.
\end{lemma}

\begin{proof}
Let $F=\{0,2\}^n\leq\mathbb{Z}_4^n$ and $R=\{0,1\}^n\subseteq \mathbb{Z}_4^n$. Every element $a\in\mathbb{Z}_4^n$ can be written as $a=f+r$ ($f\in F,r\in R$) in a unique way. Let $A\subseteq \mathbb{Z}_4^n$. Let us assign to every $x=2r\in F$ (where $r\in R$) a subset $A(x)\subseteq F$ in the following way: $A(x)=\{y\in F: r+2y\in A\}$. Three distinct elements $a_1=f_1+r_1,a_2=f_2+r_2,a_3=f_3+r_3$ (where $f_i\in F, r_i\in R$) form an arithmetic progression (in this order) if and only if $a_1+a_3=2a_2$, that is, if $f_1+f_3+r_1+r_3=2r_2$. As $f_1,f_3,2r_2\in F$, this implies $r_1=r_3$, so the condition gives $2r_2=2r_1+f_1+f_3$. Such elements exist if and only if for distinct $x=2r_1,y=2r_2\in F$ we have $y\in x+A(x)\hat{+}A(x)$ and $A(y)\ne \emptyset$. Note that $F\cong \mathbb{F}_2^n$, and this is equivalent with the condition that the system of subsets satisfies property $(*)$. Furthermore, $|A|=\sum |A(x)|$, so the maximal possible size of a 3AP-free subset of $\mathbb{Z}_4^n$ is equal to the maximal possible total size of a system of subsets $A(x)$ satisfying property $(*)$.
\end{proof}

\subsection{3AP-free sets: lower bound and subspace version}

In this subsection, first, as an illustration, we give an alternative -- different from the proof presented in Section~\ref{secproof} --  proof (using the subset reformulation) for Theorem~\ref{result1}, then we prove Theorem~\ref{subspace_thm}.

\begin{proof}[Alternative proof of Theorem~\ref{result1}.]
For $x\in \mathbb{F}_2^n$ let $\text{supp}(x)=\{i: x_i\ne 0\}$. Let us fix some $r\in \{0,1,\dots,n\}$. Let $A(x)=\{v: \text{supp}(v)\subseteq
\text{supp}(x) \}$ if $|\text{supp}(x)|=r$ and $A(x)=\emptyset$ otherwise. We claim that the system of subsets $A(x)$ satisfies $(*)$.
Indeed, if $y\in x+A(x)\hat{+}A(x)$, then $|\text{supp}(x)|=r$, thus $\text{supp}(y)\subsetneq \text{supp}(x)$ yields $A(y)=\emptyset$.

The total size of the subsets $A(x)$ is $\binom{n}{r}2^r$. The optimal choice is $r=\lceil 2n/3 \rceil$ yielding $r_3(\mathbb{Z}_4^n)\geq \binom{n}{r} 2^r \gg 3^n/\sqrt{n}$.

\end{proof}

\begin{proof}[Proof of Theorem~\ref{subspace_thm}.]
For $0\leq k\leq n$ let $X_k$ contain those $x$ for which $A(x)$ 
is a subspace of codimension $k$. If there is an $A(x)$ of codimension 0, that is, $A(x)=\mathbb{F}_2^n$, then all the other $A(y)$ sets are empty, thus the total size of the subsets is only $2^n$. From now on, we assume that each nonempty subset is a subspace of positive codimension.

Let us fix $k$.
For $x\in X_k$ let $x^{(1)},\dots,x^{(k)}$ be a basis for the orthogonal 
complement of $A(x)$, that is, 
$A(x)=\{ z: \forall\ 1\leq i\leq k:\ zx^{(i)}=0 \}$.

Let $\hat{x}=(x,1)\in\mathbb{F}_2^{n+1}$ and $\hat{x}^{(i)}=(x^{(i)},1+xx^{(i)})\in\mathbb{F}_2^{n+1}$. Now, for every $x\in X_k$ we have
$\hat{x}\hat{x}^{(i)}=1$. If $x\ne y\in X_k$, then $y\notin x+A(x)\hat{+}A(x)$, thus for some $1\leq i\leq k$ we have $(x+y)x^{(i)}=1$.
However, this implies that $(\hat{x}+\hat{y})\hat{x}^{(i)}=1$, that is, $\hat{y}\hat{x}^{(i)}=0$. Let $u(x)= \hat{x} \otimes \hat{x}
\otimes \dots \otimes \hat{x}\in (\mathbb{F}_2^{n+1})^{\otimes k}$ and $v(x)=\hat{x}^{(1)} \otimes \hat{x}^{(2)} \otimes \dots \otimes
\hat{x}^{(k)}\in (\mathbb{F}_2^{n+1})^{\otimes k}$. If $x,y\in X_k$, then $u(x)v(y)=\delta_{xy}$, so the vectors $(u(x),v(x))$ (with $x\in
X_k$) form a biorthogonal system of vectors, specially, the $u(x)$ vectors are linearly independent. However, all the $u(x)$ vectors lie in
a subspace of dimension $\sum\limits_{i=1}^k\binom{n+1}{i}$, thus $|X_k|\leq \sum\limits_{i=1}^k\binom{n+1}{i}$. Therefore, the total size
of the subsets $A(x)$ is at most $\sum\limits_{k=1}^n \sum\limits_{i=1}^k\binom{n+1}{i}2^{n-k} \leq 6\cdot 3^{n}$.

Now we use the tensor power trick to get rid of the factor 6. Let us assume that in $\mathbb{F}_2^n$ the system of subsets $A(x)$ satisfies
$(*)$ and all the non-empty subsets are subspaces. Let $S=\sum |A(x)|$. Now, we can define a system of subsets in $\mathbb{F}_2^{nt}$ as
follows. For $(x_1,x_2,\dots,x_t)\in \mathbb{F}_2^{nt}$ let $A((x_1,x_2,\dots,x_t))=A(x_1)\times A(x_2)\times \dots \times A(x_t)$. It is
easy to check that this system satisfies $(*)$, all the non-empty subsets are subspaces and the total size of the subspaces is $S^t$.
Therefore, $S^t\leq 6\cdot 3^{nt}$, thus $S\leq 6^{1/t}3^n$. This holds for every $t$, so the statement is proven. 

\end{proof}

\subsection{Reformulation for 4AP-free-ness}
Let us say that a system of subsets $A(x)\subseteq \mathbb{F}_2^n$ ($x\in\mathbb{F}_2^n$) satisfies property $(**)$, if 
the following implication holds:
$$
\forall x,y \in \mathbb{F}_2^n  \ (x+y\in (A(x)+A(x))\cap (A(y)+A(y)) \implies x=y) \eqno{(**)}
$$
(Note that for $A(x)=\emptyset$ we define $A(x)+A(x):=\emptyset$.)
Let $r_4'(n)$ denote the maximal possible size of $\sum \limits_{x\in\mathbb{F}_2^n} |A(x)|$, if the system of subsets $\{A(x):x\in \mathbb{F}_2^n\}$ satisfies $(**)$.

\begin{lemma}\label{4apeq}
For every $n\geq 1$ we have $r_4(\mathbb{Z}_4^n)=r_4'(n)$.
\end{lemma}

\begin{proof}
Similarly to the proof of Lemma~\ref{3apeq} let us write every element $a\in\mathbb{Z}_4^n$ in the form $a=f+r$ (where $f\in F:=\{0,2\}^n,r\in R:=\{0,1\}^n$). Let $A\subseteq \mathbb{Z}_4^n$. Let us assign to every $x=2r\in F$ (where $r\in R$) a subset $A(x)\subseteq F$ in the following way: $A(x)=\{y\in F: r+2y\in A\}$. 

Now four distinct elements $a_1=f_1+r_1,a_2=f_2+r_2,a_3=f_3+r_3,a_4=f_4+r_4$ (where $f_i\in F, r_i\in R)$ form an arithmetic progression (in this order) if and only if $a_1+a_3=2a_2$ and $a_2+a_4=2a_3$, that is, if $f_1+f_3+r_1+r_3=2r_2$ and $f_2+f_4+r_2+r_4=2r_3$. This implies $r_1=r_3$ and $r_2=r_4$, so the condition gives $2r_2=2r_1+f_1+f_3$ and $2r_1=2r_2+f_2+f_4$. Such elements exist if and only if for distinct elements $x=2r_1,y=2r_2\in F$  we have $y\in x+A(x)\hat{+}A(x)$ and $x\in y+A(y)\hat{+}A(y)$. Note that $F\cong \mathbb{F}_2^n$. Hence, $A$ is 4AP-free if and only if the system of subsets $\{A(x):x\in F\}$ satisfies property $(**)$.

Furthermore, $|A|=\sum |A(x)|$, so the maximal possible size of a progression-free subset of $\mathbb{Z}_4^n$ is the same as the maximal possible total size of a family of subsets $A(x)$ satisfying property $(**)$.
\end{proof}

\section{Construction, Proof of Theorem \ref{const} }\label{sec:const}

\begin{proof}[Proof of Theorem \ref{const}]
According to Lemma~\ref{3apeq} it suffices to show that $r_3'(n)\geq \sum\limits_{i=t+1}^n \binom{n}{i}C(i,i-t)$ for every $0\leq t\leq n$. That is, our aim is to find a system of subsets $A(x)\subseteq \mathbb{F}_2^n$ ($x\in \mathbb{F}_2^n$) satisfying $(*)$ and having total size $\sum |A(x)|=\sum\limits_{i=t+1}^n \binom{n}{i}C(i,i-t)$.

Let $A(x)=\emptyset$ if and only if the Hamming-weight of $x\in \mathbb{F}_2^n$ is at most $t$, that is:
$$\{x\in\mathbb{F}_2^n: A(x)=\emptyset\}=\{x\in\mathbb{F}_2^n: |x|\leq t\}=:T.$$
To guarantee the requirement $(*)$, for every $y\in \mathbb{F}_2^n\setminus T$ we have to choose the subset $A(y)$ in such a way that $y+A(y)\hat{+}A(y)\subseteq T$. Let us assume that the Hamming-weight of $y$ is $i$ (where $i\in [t+1,n]$). Let $I(y)=\{j: y_j=1\}$ be the support of $y$. Let us consider the $i$-dimensional subspace $V(y)=\{z\in \mathbb{F}_2^n: z_j=0\text{ for every }j\notin I(y) \}$. According to the definition of $C(i,i-t)$ there exists a code with minimal distance at least $i-t$ and size $C(i,i-t)$ in $V(y)$. Let $A(y)$ be the set of the codevectors (having 0 coordinates for every $j\notin I(y)$) of this code: $A(y)\subseteq V(y)\subseteq \mathbb{F}_2^n$.
If $z_1$ and $z_2$ are distinct elements of $A(y)$, then the support of each of them is a subset of $I(y)$: $I(z_1),I(z_2)\subseteq I(y)$. Furthermore, the Hamming-weight of their sum $z_1+z_2$ is at least $i-t$, since the minimum distance of the code is at least $i-t$. Therefore, the Hamming-weight of $y+z_1+z_2$ is at most $i-(i-t)=t$, which implies that $y+z_1+z_2\in T$, as needed. Hence, the system of subsets defined this way satisfies $(*)$. Also, the total size of the subsets $A(y)$ is $\sum\limits_{i=t+1}^n \binom{n}{i}C(i,i-t)$, as required.
\end{proof}

\section{3AP-free subsets of $\mathbb{Z}_4^n$, if $n \leq 4$}\label{sec3ap1}

Now, we are ready to prove Theorem~\ref{3APthm}. In this section we give a proof for $n\leq 4$,  the case $n=5$ is covered in the next section.

Before starting the proof we give a brief outline of the main strategy. If we take a look at condition $(*)$ or $(**)$, then heuristically it seems to be a good idea to use sets with small doubling, since $(*)$ and $(**)$ seem to be less restrictive for  sets with a small doubling. Subspaces have a small doubling, and working with them is easier, an important step will be to show that it can be assumed (up to $n\leq 5$) that in a maximal configuration all the (non-empty) subsets are subspaces. To arrive at this all-subspace state, we can use arguments of the following type. If $A(x)+A(x)\supseteq V$ for a large subspace $V$ (where ``large'' means that  $|V|\geq |A(x)|$), then we can replace $A(x)$ by $V$, since $(*)$ (or $(**)$) remains true (that is, the corresponding subset is still 3AP/4AP-free) and the total size of the subsets is larger (not smaller). So the general plan is to replace the subsets with subspaces, and then solve the subspace version of the problem. If the dimension is small, then for almost all subsets $A(x)$ we can do this reduction step easily, there are just a few cases, when $A(x)+A(x)$ does not contain a sufficiently large subspace. However, even in these exceptional cases $A(x)+A(x)$ turns out to be too large, so these cases can be excluded, as well. As the dimension increases, both the reduction step and both handling the all-subspace problem is getting more difficult. The 5-dimensional case is considerably more difficult than the previous cases, the proof of it is presented in the next section. Now, we continue with the proof of the cases $1\leq n\leq 4$.

\bigskip

\begin{proof}[Proof of Theorem~\ref{3APthm} in the cases $n\leq 4$.]

According to Lemma~\ref{3apeq} and Corollary~\ref{gen_lower} it suffices to show that $r_3'(1)\leq2,r_3'(2)\leq6,r_3'(3)\leq16,r_3'(4)\leq42$.
\bigskip

{\bf Case 1: $n=1$.}
If the dimension is 1, then it is trivial that every 2-element subset of $\mathbb{Z}_4$ is 3AP-free and any three elements form a 3AP, so $r_3(\mathbb{Z}_4)=2$.


We continue with some general observations that are going to be used when the dimension is at least 2. Let us take a system of subsets $A(x)(\subseteq\mathbb{F}_2^n)$ (indexed by elements $x\in \mathbb{F}_2^n$) satisfying $(*)$. For brevity let $S=\sum\limits_{x\in \mathbb{F}_2^n} |A(x)|$.

{\bf Observation 1.} If $2^{n-1}<|A(x)|$ for some $x\in\mathbb{F}_2^n$, then by the pigeon-hole principle $x+A(x)+ A(x)=\mathbb{F}_2^n$. Since, for every $y\in\mathbb{F}_2^n$ we have $(x+A(x))\cap (y+A(x))\ne \emptyset$, so, for some $a_1,a_2\in A(x)$ we have $x+a_1=y+a_2$, that is, $y=x+a_1+a_2\in x+A(x)+A(x)$. Therefore, $x+A(x)\hat{+}A(x)=\mathbb{F}_2^n\setminus \{x\}$, so all the subsets are empty except $A(x)$, thus $S=|A(x)|\leq 2^n$. Hence, in this case the statement holds.

From now on, let us assume that $|A(x)|\leq 2^{n-1}$ for every $x$.

{\bf Observation 2.}
Let $A(x)$ be a nonempty subset: $0<|A(x)|\leq 2^{n-1}$.
It can be assumed that $0\in A(x)$, since changing $A(x)$ to a translate of itself, $A(x)+c$, preserves $A(x)+A(x)$.

{\bf Observation 3.} If $|A(x)|\in\{1,2\}$, then $A(x)$ is automatically a subspace, as $0\in A(x)$.
If $|A(x)|\in\{3,4\}$, let $u$ and $v$ be two different nonzero elements of $A(x)$, that is, $A(x)\supseteq \{0,u,v\}$. Clearly, for $A'(x)=\langle u,v \rangle$ we have $A(x)\hat{+} A(x) \supseteq A'(x)\hat{+}A'(x)$, so we may replace $A(x)$ by the 2-dimensional linear subspace $A'(x)$. This way $(*)$ is still satisfied, and either $S$ does not change or it increases by 1.

Now we consider the cases $n=2,3,4$ one by one.

\bigskip

{\bf Case 2: $n=2$.}

Now, we continue with the case when the dimension is 2. If none of the subsets is empty, then all of them can have size at most 1, thus $S\leq 4$. Otherwise, by Observation 1 we can assume that every nonempty subset has size at most 2,
thus $S\leq 6$, since there must be an empty set. 


\bigskip

{\bf Case 3: $n=3$.}

If the dimension is 3, then let $e_1,e_2,e_3$ be a basis for $\mathbb{F}_2^3$.

According to Observations 1-3 we can assume that all subsets have size at most 4 and every nonempty subset is a subspace (of dimension at most 2).

Let $k$ denote the number of 2-subspaces and $l$ the number of empty sets. If $k=0$, then $S\leq 2\cdot 8=16$, and we are done. Note that in fact $S<16$, since either all subsets have size at most 1 or at least one of them is empty.

So we can assume that $k>0$. If $A(x)=\langle u,v\rangle$ a 2-subspace, then $A(x+u),A(x+v),A(x+u+v)$ are all empty, that is, we can assign an ``empty triple'' $\{x+u,x+v,x+u+v\}$ to each 2-subspace. To different 2-subspaces we assign different triples, as the sum of the elements in the triple is $x$. 
That is, $k\leq \binom{l}{3}$. We have $S\leq 4k+2(8-k-l)=16+2k-2l\leq 16+2\binom{l}{3}-2l\leq 16$, if $l\leq 4$, equality holds if and only if $l=4$. 
If $5\leq l$, then $S\leq 3\cdot 4= 12$. Therefore, $S\leq 16$ is shown and the maximum occurs when $k=l=4$.



We continue with the 4-dimensional case.

\bigskip

{\bf Case 4: $n=4$.}

We will show that if the system of subsets  $\{ A(x)\subseteq \mathbb{F}_2^4\ |\ x\in\mathbb{F}_2^4 \}$ satisfies $(*)$, then $\sum\limits_{x\in\mathbb{F}_2^4} |A(x)|\leq 42$.

At first it is going to be shown that ``in most of the cases'' it can be assumed that all the nonempty $A(x)$ subsets are linear subspaces, then we will prove the statement for the special case when the non-empty $A(x)$ subsets are all linear subspaces and finally we will also cover the remaining cases. 

By Observations 1-3 we can assume that all subsets have size at most 8 and every nonempty subset of size at most 4 is a subspace (of dimension at most 2).

Let $5\leq |A(x)|\leq 8$. As $\dim \langle A(x) \rangle \geq 3$, we may choose three linearly independent vectors from $A(x)$. Let these be $f_1,f_2,f_3$ and let $A'(x)=\langle f_1,f_2,f_3\rangle$. As $0,f_1,f_2,f_3\in A(x)$, we have that $\{0,f_1,f_2,f_3,f_1+f_2,f_1+f_3,f_2+f_3\}\subseteq A(x)+A(x)$, that is, $A(x)+A(x)$ contains all the elements of the subspace $\langle f_1,f_2,f_3 \rangle$, possibly with the exception of $f_1+f_2+f_3$.

We claim that if there exists some $0\ne g\in  (A(x)\cap A'(x))\setminus \{f_1,f_2,f_3\}$, then $A(x)+ A(x)\supseteq \langle f_1,f_2,f_3 \rangle$. To see this, we only need to show that $f_1+f_2+f_3\in A(x)\hat{+} A(x)$. However, either $g=f_i+f_j$ (with some distinct $i,j\in \{1,2,3\}$) and $f_1+f_2+f_3=g+f_k$ (where $\{i,j,k\}=\{1,2,3\}$) or $g=f_1+f_2+f_3$ and $f_1+f_2+f_3=g+0$ is a good representation.
Therefore, in this case we can replace $A(x)$ by $\langle f_1,f_2,f_3 \rangle$. It remains to check the case when any four vectors in $A(x)\setminus \{0\}$ are linearly independent.

{\bf Step 1.} Assuming that $A(x)$ is not a subspace, and any four vectors in $A(x)\setminus \{0\}$ are linearly independent we prove $S<42$ under the additional assumption that at most two subsets have size 8.

 Without loss of generality it can be assumed that $\{0,f_1,f_2,f_3,f_4\}\subseteq A(x)$, where $f_1,f_2,f_3,f_4$ is a basis. The 3-subspaces spanned by three out of these basis vectors cover $\mathbb{F}_2^4$ with the exception of $f_1+f_2+f_3+f_4$. That is, if $|A(x)|\ne 5$, then $A(x)=\{0,f_1,f_2,f_3,f_4,f_1+f_2+f_3+f_4\}$, but in this case $A(x)\hat{+}A(x)\supseteq A'(x)\hat{+} A'(x)$ for $A'(x)=\langle f_1,f_2,f_3 \rangle$, so we can replace $A(x)$ by a larger set $A'(x)$. So, it suffices to check the case when $A(x)=\{0,f_1,f_2,f_3,f_4\}$. The system of subsets $\{A(y)\ |\ y\in \mathbb{F}_2^4\}$ can be replaced by a ``translate'' of itself: $\{A'(y)\ |\ y\in \mathbb{F}_2^4\}$ where $A'(y)=A(y+c)$ for some fixed $c\in \mathbb{F}_2^4$ (not depending on $y$). So by taking $c=x$ we may suppose that $A(0)=\{0,f_1,f_2,f_3,f_4\}$. Then $|0+A(0)\hat{+}A(0)|=10$, so at least 10 subsets are empty. The size of $A(0)$ is 5 and the size of the other five (possibly) nonempty subsets is at most 8. If at least two out of these five subsets have size at most 5, then $S\leq 5+5+5+3\cdot 8=39<42$. If this does not hold, then at least four of them are of size 8. We will cover this case later: indeed, it is going to be shown that if at least three subsets are of size 8, then $S<42$.

{\bf Step 2.}
From now on, we will assume that 
\begin{itemize}
\item either all the nonempty $A(x)$ sets are linear subspaces of dimension at most 3, or
\item some of the subsets are of size 5 but there are at least three subsets of size 8,
\end{itemize}
and show that $S\leq 42$ in these cases, too.

Let $h$ be the number of 3-subspaces. We distinguish 4 subcases.

\begin{itemize}
\item[Subcase 1 ($h=0$)]
In this case all of the subsets are of size at most 4. If $A(x)=\langle u,v\rangle$ is a 2-dimensional subspace for some $x$, then $A(x)\hat{+}A(x)=\{u,v,u+v\}$, thus $A(x+u),A(x+v)$ and $A(x+u+v)$ are all empty. So for each 2-subspace $A(x)$ we can assign an ``empty triple'', since the subsets assigned to the elements of $x+A(x)\hat{+}A(x)$ are all empty. Moreover, the triple $\{x+u,x+v,x+u+v\}$ determines $x$, since the sum of the vectors in the triple is $x$.  Let $k$ be the number of 2-subspaces and $l$ be the number of empty subsets (among the $A(x)$ sets). As empty triples can be assigned to the 2-subspaces by an injective mapping, we have $k\leq \binom{l}{3}$.\\
Hence, $S\leq 4k+2(16-k-l)=32+2k-2l\leq 32+2\binom{l}{3}-2l$. If $l\leq 4$, then this yields $S\leq 32$. Moreover, for $l=5$ we obtain that $S\leq 42$.  
If $l\geq 6$, then $S\leq 10\cdot 4=40$.\\
In all cases we obtained that $S\leq 42$.

\item[Subcase 2 ($h=1$)]
Let $|A(0)|=8$. As $|0+A(0)\hat{+}A(0)|=7$, at least 7 subsets are empty and consequently $S\leq 8+(16-1-7)\cdot 4=40$.

\item[Subcase 3 ($h=2$)]
Let $A(u)$ and $A(v)$ be the two 3-subspaces. Then $U=u+A(u)+A(u)$ and $V=v+A(v)+A(v)$ are 3-dimensional affine subspaces. If $U\cap V=\emptyset$, then $U\cup V=\mathbb{F}_2^4$ and $A(x)=\emptyset$ for all $x\notin\{ u,v\}$, so $S\leq 2\cdot 8=16$. Otherwise, $U\cap V$ is a 2-dimensional affine subspace, so $|(U\cup V)\setminus \{u,v\}|=(16-4)-2=10$, that is, at least 10 subsets are empty. Then $S\leq 2\cdot 8+4\cdot 4=32$.

\item[Subcase 4 ($h\geq 3$)]
Finally, let us assume that $A(u),A(v),A(w)$ are 3-subspaces. Note that in this case it can happen that some of the nonempty subsets  are not subspaces (these sets have size 5 and contain 5 affine independent vectors). According to Subcase 3, at least 10 subsets are empty. If at least 11 subsets are empty, then $S\leq 5\cdot 8=40$, and we are done. So it can be assumed that exactly 10 subsets are empty. Let $U=u+A(u)+A(u),V=v+A(v)+A(v),W=w+A(w)+A(w)$. Since there are only 10 empty subsets, from the argument of Subcase 3 it follows that these are exactly the 10 subsets $A(x)$ which are assigned to the 10 elements $x\in (U\cup V)\setminus\{u,v\}$.     
However, $U, V, U\cap V$ are all affine subspaces, so the sum of the vectors in $U$ adds up to 0 and the same holds for $V$ and $U\cap V$. Thus the sum of the vectors in $U\cup V$ is also 0. Hence, the sum of all vectors to which the empty set is assigned is $u+v$. However, we can repeat this argument with $U$ and $W$ and get that the sum is also equal to $u+w$, which is a contradiction. We are done. 

\end{itemize}

\end{proof}

\begin{proof}[Proof of Theorem~\ref{unique}]
We are going to use the implications of the previous proof.

When $n=2$, one of the sets must be empty and all other sets must have size 2 in order to get 6 elements. If, say, $A(x_0)=\emptyset$, then for any $x\ne x_0$ the set $A(x)$ must contain two elements whose difference is $x$. Two such configurations always can be mapped to each other in the required way.

When $n=3$, then we need four empty sets and four 2-subspaces to get the total size of 16. Assume that $A(x_1)=A(x_2)=A(x_3)=A(x_4)=\emptyset$. We claim that $x_1,x_2,x_3,x_4$ are affine independent. Otherwise they form an affin 2-subspace, however, taking some $x\notin \{x_1,x_2,x_3,x_4\}$ the affine 2-subspace $x+A(x)+A(x)$ would have to contain exactly three of $x_1,x_2,x_3,x_4$ (and $x$ as the fourth element) which is impossible. Therefore, $x_1,x_2,x_3,x_4$ are affine independent, and by some affine linear transformation $\varphi$ these can be mapped to $0,e_1,e_2,e_3$, for simplicity. Now, we can assume that 0 is contained in every nonempty $A(x)$ (by suitable translations). Then it follows that $A(e_i+e_j)=\langle e_i,e_j\rangle$, for $1\leq i<j\leq 3$ and $A(e_1+e_2+e_3)=\langle e_1+e_2,e_2+e_3\rangle$.

Finally, let $n=3$. 
Note that $S=42$ can hold only in Subcase 1 when $k=10,l=5$.

From the proof it follows that $S=42$ is possible only if there are exactly five empty sets, ten 2-subspaces and one 1-subspace. Moreover, if $u_1,u_2,u_3,u_4,u_5$ are the vectors to which the empty set is assigned, then the 3-term sums made out of these 5 vectors have to be all distinct. Clearly, by applying a suitable affine linear transformation $\varphi$ we can assume that $u_1=0$ and  $u_2,u_3,u_4$ are linearly independent. If $u_5\in\langle u_2,u_3,u_4\rangle$, then all the 10 triple sums lie in a 3-subspace, so they can not be all distinct. Thus $u_2,u_3,u_4,u_5$ are linearly independent. Therefore, by renaming  $u_1,\dots,u_5$ (if necessary), let $A(0)=A(e_1)=A(e_2)=A(e_3)=A(e_4)=\emptyset$, where $e_1,e_2,e_3,e_4$ is a basis. The set $A(e_1+e_2+e_3+e_4)$ can not be a 2-subspace, since all vectors in it must have Hamming-weight at least 3 to satisfy $e_1+e_2+e_3+e_4+A(e_1+e_2+e_3+e_4)\hat{+}A(e_1+e_2+e_3+e_4)\subseteq\{0,e_1,e_2,e_3,e_4\}$. So it is the unique 1-subspace, for instance $A(e_1+e_2+e_3+e_4)=\langle e_1+e_2+e_3+e_4 \rangle$ is an appropriate choice, but $\langle e_i+e_j+e_k\rangle$ is also fine with any 3-subset $\{i,j,k\}$ of $\{1,2,3,4\}$.  
By permuting $0,e_1,e_2,e_3,e_4$ with a suitable affine linear transformation we might assume that $A(e_1+e_2+e_3+e_4)=\langle e_1+e_2+e_3+e_4 \rangle$.

The remaining 10 sets need to be 2-subspaces. For $A(e_i+e_j)$ the unique appropriate choice is $A(e_i+e_j)=\langle e_i,e_j\rangle$, with this choice $e_i+e_j+A(e_i+e_j)\hat{+}A(e_i+e_j)=\{0,e_i,e_j\}$ holds. For $A(e_i+e_j+e_k)$ the unique appropriate choice is $A(e_i+e_j+e_k)=\langle e_i+e_j,e_i+e_k \rangle =\{0,e_i+e_j,e_j+e_k,e_k+e_i\}$, with this choice $e_i+e_j+e_k+A(e_i+e_j+e_k)\hat{+}A(e_i+e_j+e_k)=\{e_i,e_j,e_k\}$ is satisfied. 

\end{proof}


\section{Proof of $r_3(\mathbb{Z}^5_4)=124$}
\label{sec3ap2}

We will show that if the system of subsets  $\{ A(x)\subseteq \mathbb{F}_2^5\ |\ x\in\mathbb{F}_2^5 \}$ satisfies $(*)$, then $S:=\sum\limits_{x\in\mathbb{F}_2^5} |A(x)|\leq 124$.

Again, by Observations 1-3 we can assume that all subsets have size at most 16 and every nonempty subset of size at most 4 is a subspace (of dimension at most 2).



Now, let us assume that $8<|A(x)|\leq 16$. The set $A(x)$ must contain at least 4 linearly independent vectors. (Note that by Observation 2 we have $0\in A(x)$.)

\smallskip
{\bf Step 1.} First, let us assume that a set $A(x)$ with size $8<|A(x)|\leq 16$ spans a 4-dimensional subspace. Our aim is to show it can be assumed that $A(x)$ itself is a 4-subspace.
\smallskip

Let $f_1,f_2,f_3,f_4\in A(x)$ be linearly independent. Then $A(x)\hat{+}A(x)$ contains all the pairwise sums $f_i+f_j$. If $f_1+f_2+f_3+f_4$ also lies in $A(x)$, then $A(x)+A(x)=\langle f_1,f_2,f_3,f_4\rangle$, since the 3-term sums like $f_1+f_2+f_3$ can be obtained as $(f_1+f_2+f_3+f_4)+f_4=f_1+f_2+f_3$ and $f_1+f_2+f_3+f_4=(f_1+f_2+f_3+f_4)+0\in A(x)\hat{+}A(x)$. Hence, if $f_1+f_2+f_3+f_4\in A(x)$, then $A(x)$ can be replaced by $A'(x)=\langle f_1,f_2,f_3,f_4\rangle$.

Now, let us assume that $f_1+f_2+f_3+f_4\notin A(x)$.
Let us call the 2-term sums  $f_i+f_j$ (with $i\ne j$) {\it pairs} and the 3-term sums $f_i+f_j+f_k$ (with $i,j,k$ distinct) {\it triples}. The pair $f_i+f_j$ can be identified with the set of indices $\{i,j\}$, let us call this subset $\{i,j\}\subseteq\{1,2,3,4\}$ also a {\it pair}, and similarly the 3-element subset $\{i,j,k\}$ will be called a {\it triple} corresponding to the vector $f_i+f_j+f_k$.
 As the size of $A(x)$ is at least 9, the set $A(x)$ must contain at least $(9-4-1=)4$ elements among the six pairs and four triples.

As a first observation we check that in all of the following cases the equality $A(x)+A(x)=\langle f_1,f_2,f_3,f_4\rangle$ holds:
\begin{itemize}
\item[(i)] $A(x)$ contains two disjoint pairs: for instance $f_1+f_2,f_3+f_4\in A(x)$,
\item[(ii)] $A(x)$ contains a pair and a triple such that their intersection has size 1: for instance: $f_1+f_2$ and $f_2+f_3+f_4$,
\item[(iii)] $A(x)$ contains all the 3-term sums.
\end{itemize}

In case (i) we have $f_1+f_2+f_3+f_4=(f_1+f_2)+(f_3+f_4)$ and each triple contains either $\{1,2\}$ or $\{3,4\}$, thus they can be expressed like $f_1+f_2+f_3=(f_1+f_2)+f_3$.

In case (ii) we have $f_1+f_2+f_3+f_4=f_1+(f_2+f_3+f_4)$, the triples $\{1,2,3\}$ and $\{1,2,4\}$ can be obtained like $f_1+f_2+f_3=(f_1+f_2)+f_3$, furthermore, $f_2+f_3+f_4=0+(f_2+f_3+f_4)$ and $f_1+f_3+f_4=(f_1+f_2)+(f_2+f_3+f_4)$, as all $f_1,f_2,f_3,f_4\in A(x)$.

In case (iii) all the triples can be written like $f_1+f_2+f_3=(f_1+f_2+f_3)+0$ and $f_1+f_2+f_3+f_4=(f_1+f_2+f_3)+f_4$.


Now let us assume that none of (i-iii) holds. Since we need at least four more vectors, at least one triple is contained in $A(x)$, by symmetry we shall assume that $f_1+f_2+f_3\in A(x)$. Then $f_1+f_2+f_3+f_4=(f_1+f_2+f_3)+f_4\in A(x)\hat{+}A(x)$.
Also, $A(x)$ must contain at least one pair. This pair must be a subset of $\{1,2,3\}$, otherwise it would intersect $\{1,2,3\}$ in a single element, contradicting that (ii) does not hold. We can assume that $f_1+f_2\in A(x)$. Note that $f_1+f_3+f_4$ and $f_2+f_3+f_4$ are not in  $A(x)$, since (ii) does not hold. Now, $f_1+f_2+f_3=(f_1+f_2+f_3)+0\in A(x)\hat{+}A(x)$ and $f_1+f_2+f_4=(f_1+f_2)+f_4\in A(x)\hat{+}A(x)$. We claim that $f_1+f_3+f_4$ and $f_2+f_3+f_4$ also lie in $A(x)\hat{+}A(x)$. 
Since (ii) does not hold, the pairs $f_1+f_4,f_2+f_4,f_3+f_4$ are not in $A(x)$. Therefore, to get at least 9 elements we have to take at least two vectors from $\{f_1+f_3,f_2+f_3,f_1+f_2+f_4\}$. The triple $f_1+f_2+f_4$ can not be taken, since it intersects the two pairs, $f_1+f_3$ and $f_2+f_3$, in a single element. Now, $f_1+f_3,f_2+f_3\in A(x)$ implies that $f_1+f_3+f_4,f_2+f_3+f_4\in A(x)\hat{+}A(x)$, and we are done.

Thus in all cases we get $A(x)+A(x)=\langle f_1,f_2,f_3,f_4 \rangle$. Hence, if $A(x)$ is a set of size at least 9 (and at most 16) such that $A(x)$ is not a 4-subspace, then we can assume that  $\dim \langle A(x) \rangle =5$.


\smallskip
{\bf Step 2.} We show that it can be assumed that there is no subset for which $8<|A(x)|\leq 16$ and $\dim \langle A(x) \rangle =5$. Our aim is to show that $A(x)$ can be replaced by a 4-subspace. Together with Step 1 this implies that we can assume that all sets having size larger than 8 are 4-subspaces. Moreover, we show that there can be at most one such subset.
\smallskip

Our aim is to show that either there is a 4-subspace $A'(x)$ such that $A'(x)\subseteq A(x)+A(x)$ or the total size $S$ of the sets is at most 124.

Let us assume that $0,f_1,f_2,f_3,f_4,f_5\in A(x)$, where $f_1,\dots,f_5$ is a basis.
Then all singletons $f_i$ and pairs $f_i+f_j$ lie in $A(x)\hat{+}A(x)$.
If a 4-term sum, like $f_1+f_2+f_3+f_4$ lies in $A(x)$, then $A(x)+A(x)$ contains $\langle f_1,f_2,f_3,f_4\rangle$ and we are done: $A(x)$ can be replaced by $\langle f_1,f_2,f_3,f_4\rangle$. More generally we can formulate the following observation:

{\bf Observation 4.} If it is possible to choose 6 vectors $w_1,\dots,w_6$ from $A(x)$ in such a way that they span a 4-dimensional affine subspace and their sum is 0, then $A(x)$ can be replaced by a 4-subspace, since translating $A(x)$ by $w_6$ and taking $f_1=w_1+w_6,f_2=w_2+w_6,\dots, f_4=w_4+w_6$ gives $w_5+w_6=f_1+f_2+f_3+f_4$, so this case can be handled in the same way as the previous case.

Therefore, $f_1+f_2+f_3+f_4+f_5\notin A(x)$, since $\{f_1,f_2,f_3,f_4,f_5,f_1+f_2+f_3+f_4+f_5\}$ adds up to 0.
Thus the remaining elements of $A(x)$ are all pairs and triples.
We claim that the following cases can be excluded with the help of Observation 4:
\begin{itemize}
\item[(i)] there are two disjoint pairs, e.g. $f_1+f_2,f_3+f_4\in A(x)$
\item[(ii)] there are two triples intersecting each other in a single element, e.g. $f_1+f_2+f_3,f_3+f_4+f_5\in A(x)$
\item[(iii)] there is a pair and a triple intersecting each other in a single element, e.g. $f_1+f_2,f_2+f_3+f_4\in A(x)$
\end{itemize}  

In case (i) $f_1+f_2+f_3+f_4+(f_1+f_2)+(f_3+f_4)=0$.

In case (ii) $(f_1+f_2+f_3)+(f_3+f_4+f_5)+f_1+f_2+f_4+f_5=0$.

In case (iii) $(f_1+f_2)+(f_2+f_3+f_4)+f_1+f_3+f_4+0=0$.

Finally, let us assume that (i-iii) do not hold. From (i) it follows that the pairs either form a star or a triangle.
If they form a triangle, let us assume that it is $f_1+f_2,f_2+f_3,f_1+f_3$. Since  $f_1+f_2+f_3\in A(x)$ would imply $\langle f_1,f_2,f_3,f_4\rangle\subseteq A(x)+A(x)$, we have $f_1+f_2+f_3\notin A(x)$. Furthermore, (iii) implies that none of the other triples is in $A(x)$. Hence, $A(x)=\{0,f_1,f_2,f_3,f_4,f_5,f_1+f_2,f_2+f_3,f_1+f_3\}$, we will refer to this as case (a). From now on, we assume that the pairs in $A(x)$ form a star.

If this star contains 4 vectors, e.g. $f_1+f_2,f_1+f_3,f_1+f_4,f_1+f_5\in A(x)$, then $A(x)$ can not contain any triples because of (iii). (Case (b).)

If this star contains 3 vectors, e.g. $f_1+f_2,f_1+f_3,f_1+f_4$, then $A(x)$ can not contain any triples because of (iii). (Case (c).)

If this star contains 2 vectors, e.g. $f_1+f_2,f_1+f_3$. At least one triple must lie in $A(x)$ and (iii) implies that this triple is $f_1+f_2+f_3$. (Case (d).)

If only one pair is in $A(x)$, e.g. $f_1+f_2\in A(x)$. There are at least two more vectors (thus triples) in $A(x)$. If one of them is $f_3+f_4+f_5$, then the other triple intersects the pair $\{1,2\}$ or the triple $\{3,4,5\}$ in one element, contradicting (ii) or (iii). Thus, by (iii) these two triples must contain $\{1,2\}$, which gives case (e).

If there are no pairs, then there are at least three triples. Any two of them have an intersection of size 2, giving case (f) or case (g).

We summarize this:
\begin{itemize}
\item[(a)] $A(x)=\{0,f_1,f_2,f_3,f_4,f_5,f_1+f_2,f_2+f_3,f_1+f_3\}$
\item[(b)] $A(x)=\{0,f_1,f_2,f_3,f_4,f_5,f_1+f_2,f_1+f_3,f_1+f_4,f_1+f_5\}$
\item[(c)] $A(x)=\{0,f_1,f_2,f_3,f_4,f_5,f_1+f_2,f_1+f_3,f_1+f_4\}$
\item[(d)] $A(x)=\{0,f_1,f_2,f_3,f_4,f_5,f_1+f_2,f_1+f_3,f_1+f_2+f_3\}$
\item[(e)] $A(x)=\{0,f_1,f_2,f_3,f_4,f_5,f_1+f_2,f_1+f_2+f_3,f_1+f_2+f_4\}$
\item[(f)] $A(x)=\{0,f_1,f_2,f_3,f_4,f_5,f_1+f_2+f_3,f_1+f_2+f_4,f_1+f_2+f_5\}$
\item[(g)] $A(x)=\{0,f_1,f_2,f_3,f_4,f_5,f_1+f_2+f_3,f_1+f_2+f_4,f_1+f_3+f_4\}$
\end{itemize}

Note that the size of $A(x)$ is 10 in case (b) and 9 in the remaining cases (a) and (c-g). Also, the size of $A(x)\hat{+}A(x)$ is 21 in cases (b), (c), (e), (f) and 22 in cases (a), (d), (g).

Let us assume that there is at least one subset $A(x)$ having size at least 9 and not being a 4-subspace. Then at least 21 subsets out of the 32 sets $A(y)$ are empty, so at most 11 subsets are non-empty. Let $k$ denote the number of 4-subspaces among the subsets $A(x)$. Then $S=\sum |A(y)|\leq 16k+10(11-k)=110+6k$. If $k\leq 2$, then this is at most 122. So let us assume that there are at least three 4-subspaces, namely, $A(y),A(z), A(u)$. Let $K=y+A(y), L=z+A(z), M=u+A(u)$, then $K,L,M$ are affine subspaces of dimension 4.

If  two of them are disjoint, for instance $K\cap L=\emptyset$, then $K\cup L=\mathbb{F}_2^5$, giving that $A(t)=\emptyset $ for every $t\notin \{y,z\}$, which is a contradiction. So any two of them intersect nontrivially each other, and then any pairwise intersection is a 3-dimensional affine subspace. As $y\notin L\cup M$, we have that $(K\cap L)\cap (K\cap M)\ne \emptyset$, since both of them is an 8-element subset of the 15-element set $K\setminus \{y\}$, hence $K\cap L\cap M\ne \emptyset$. Then $K\cap L\cap M$ has size 4 or 8. By inclusion-exclusion principle, in both cases $|K\cup L\cup M|=|K|+|L|+|M|-|K\cap L|-|K\cap M|-|L\cap M|+|K\cap L\cap M|\geq 28$, therefore, at least $28-3=25$ subsets are empty and $S\leq 7\cdot 16=112$.

Therefore, it can be assumed that all subsets having at least 9 elements are 4-subspaces, moreover there are at most 2 such subsets. If there are 2 such subsets $A(x)$ and $A(y)$, then $|A(x)\cup A(y)|=|A(x)|+|A(y)|-|A(x)\cap A(y)|\geq 16+16-8=24$, so at least $24-2=22$ subsets are empty and $S\leq 2\cdot 16 + 8\cdot 8=96$. Hence, it can be assumed that there is at most one 4-subspace.


\smallskip
{\bf Step 3.} Now we show that if $|A(x)|\in [5,8]$, then it can be assumed that $A(x)$ is either a 3-subspace or a set of 5 or 6 affine independent points.
\smallskip

Let us assume that $4<|A(x)|\leq 8$. If $\langle A(x)\rangle$ has dimension 3, then $A(x)$ can be replaced with this 3-subspace. If $\dim \langle A(x)\rangle=4$, then it can be assumed that $0,f_1,f_2,f_3,f_4\in A(x)$. If at least one more element is in $A(x)$, then $A(x)+A(x)$ contains a 3-subspace and we can replace $A(x)$ by this 3-subspace, otherwise  $A(x)=\{0,f_1,f_2,f_3,f_4\}$, we will refer to this case as case (A).

If $\dim \langle A(x)\rangle=5$, then it can be assumed that $0,f_1,f_2,f_3,f_4,f_5\in A(x)$. If at least one more element with Hamming-weight at most 4 is in $A(x)$, then $A(x)+A(x)$ contains a 3-subspace. If $f_1+f_2+f_3+f_4+f_5\in A(x)$, then $\langle f_1+f_2,f_2+f_3,f_3+f_4\rangle \subseteq A(x)+A(x)$, otherwise $A(x)=\{0,f_1,f_2,f_3,f_4,f_5\}$, we will refer to this case as case (B).

Hence, it can be assumed that if there is a subset $A(x)$ (with size in $[5,8]$) which is not a subspace, then it contains 5 or 6 affine independent points:

\begin{itemize}
\item[(A)] $A(x)=\{0,f_1,f_2,f_3,f_4\}$
\item[(B)] $A(x)=\{0,f_1,f_2,f_3,f_4,f_5\}$
\end{itemize}
Note that the size of $A(x)$ in these cases is either 5 or 6.


\smallskip
{\bf Step 4.} We show that it can be assumed that all subsets have size at most 8.
\smallskip

Note that we have already seen (in Step 2) that there can be at most one 4-subspace, so let us assume that there exists a (unique) 4-subspace $A(y)$. Then $|A(y)\hat{+}A(y)|=15$, so there are at least 15 empty subsets. All the other subsets are 3-subspaces or have size at most 6. If there is no 3-subspace, then $S\leq 16+16\cdot 6=112$, and we are done. Let $A(x)$ be a 3-subspace and $K=y+A(y),L=x+A(x)$. As $|K\cap L|\leq 4$, we have $|K\cup L|\geq 16+8-4=20$, so there are at least $20-2=18$ empty subsets, thus at most 14 non-empty ones implying $S\leq 16+13\cdot 8=120$, and we are done. Therefore, none of the subsets can be a 4-subspace, and consequently all the subsets have size at most 8.



\smallskip
{\bf Step 5.} We show that it can be assumed that all nonempty subsets are subspaces of dimension at most 3 or a set of 5 or 6 affine independent points. Furthermore, the number of empty sets among the $A(x)$ subsets is at most 16 and there exists a subset of size at least 5.
\smallskip

If $0<|A(x)|\leq 4$, then by Observations 1-3 it can be assumed that $A(x)$ is a subspace.

Now, we can assume that all the subsets have size at most 8 and all those non-empty subsets that are not subspaces are of type (A) or (B).

If there are at least 17 empty subsets, then $S\leq 8\cdot 15=120$, so it can be assumed that at most 16 subsets are empty.

If there is no subset with size larger than 2, then $S\leq 64$.
If there is no subset with size larger than 4, then there must be a subset with size 4 and there are at most 29 non-empty sets, so $S\leq 29\cdot 4=116$. So there is a subset of size at least five, this can be either of type (A) or (B) or a 3-subspace.

Now our aim is to show that we can assume that there is no subset of type (A) neither of type (B).

\smallskip
{\bf Step 6.} We show that there is no subset of type (B).
\smallskip

Let us assume that there is a subset of type (B). Without loss of generality this is $A(0)=\{0,e_1,e_2,e_3,e_4,e_5\}$. Then $A(0)\hat{+}A(0)=\{e_1,\dots,e_5,e_1+e_2,\dots,e_4+e_5\}=:T$, that is, $A(0)\hat{+}A(0)$ has size 15 and $A(e_i)=\emptyset,A(e_i+e_j)=\emptyset$ for every $i\ne j$. As $17\cdot 6=102=124-22$, at least 11 subsets are 3-subspaces. Let $A(x)$ be a 3-subspace and $K:=x+A(x)$. As $A(0)\ne\emptyset$, we have $0\notin K$. We claim that $K\setminus \{x\} \not \subseteq T$. 

For the sake of contradiction, let us assume the contrary. Let $U=(e_1+e_2+e_3+e_4+e_5)^\perp$ and $\overline{U}=\mathbb{F}_2^5\setminus U$. As $|T\cap \overline{U}|=5$, the set $K\cap \overline{U}$ can not be  a 3-subspace. If $K\cap \overline{U}$ is a 2-subspace, then without  loss of generality, $x=e_1+e_2+e_3$ and $K\cap \overline{U}=\{e_1+e_2+e_3,e_1,e_2,e_3\}$. However, none of the translates of this set is contained in $K\cap U$, thus  we must have $K\subseteq U$, which leads to a contradiction, as well.

Hence, there exists some $y\notin T$ such that $A(y)=\emptyset$, so the number of the empty subsets is at least 16. If the number of 3-subspaces is at most 14, then $S\leq 14\cdot 8+2\cdot 6=124$, and we are done. So we can suppose that the number of 3-subspaces is at least 15 and one subset has size 6. The set $A(e_1+e_2+e_3+e_4+e_5)$ is not a 3-subspace, since any affine 3-subspace containing $e_1+e_2+e_3+e_4+e_5$ contains at least 2 more elements that are not in $T$. Hence $A(e_1+e_2+e_3+e_4+e_5)$ is the 16th empty subset. Now, we claim that $A(e_1+e_2+e_3+e_4)$ is not a 3-subspace. This holds, since any affine 3-subspace containing $e_1+e_2+e_3+e_4$ has at least one more element outside of $T\cup \{e_1+e_2+e_3+e_4+e_5\}$.

Therefore, there is no subset of type (B).

\smallskip
{\bf Step 7.} We show that there is no subset of type (A).
\smallskip

Let us assume that there is a subset of type (A), it can be assumed that it is $A(0)=\{0,e_1,e_2,e_3,e_4\}$. Then $|A(0)|=5$ and $A(0)\hat{+}A(0)=\{e_1,\dots,e_4,e_1+e_2,\dots,e_3+e_4\}$ has size 10. That is, we already have 10 empty subsets. 

For brevity let us write $A(i_1i_2\dots i_l)$ for $A(e_{i_1}+e_{i_2}+\dots+e_{i_l})$ if $$\{i_1,i_2,\dots, i_l\}\subseteq \{1,2,3,4,5\}.$$
(E.g. $A(1)=A(e_1), A(123)=A(e_1+e_2+e_3)$, and so on.)

Let us assume first that the total size of the subsets 
$$A(123),A(124),A(134),A(234),A(1234)$$
is at most 32.

Consider the following 16 subsets:  $A(z+e_5)$ ($z\in\langle e_1,e_2,e_3,e_4 \rangle$). Let $k$ denote the number of 3-subspaces among these and $l$ the number of empty ones. 
If $S\geq 125$, then $\sum |A(z+e_5)|\geq 125-5-32=88$, thus $8k+5(16-k-l)\geq 88$, and then
\begin{equation}\label{klelso}
3k\geq 5l+8.
\end{equation}


If $A(z+e_5)$ is a 3-subspace, then $K_z=z+e_5+A(z+e_5)$ is an affine 3-subspace containing $z+e_5$. The 1-codimensional affine subspace $R=\{x:xe_5=1\}$ contains either all 8 elements of $K_z$ or 4 elements of $K_z$. In the first case we get 7 new empty subsets, so the total number of empty subsets is at least 17 and we are done: $S\leq 15\cdot 8=120$. So for every 3-subspace $K_z$ exactly 4 elements of $K_z$ lie in $R$. The sum of these 4 vectors is 0, so the sum of the three vectors in $(K_z\cap R)\setminus \{z+e_5\}$ is $z+e_5$. Hence, for every 3-subspace $K_z$ we get an ``empty triple'' of vectors from $R$, therefore,
\begin{equation}\label{klmasodik}
\binom{l}{3}\geq k.
\end{equation}
By \eqref{klelso} and \eqref{klmasodik} we obtain that $l(l-1)(l-2)/2\geq 5l+8$, which yields $l\geq 6$. Then \eqref{klelso} implies that $k\geq 13$, which is a contradiction, since $6+13>16=|R|$. 


Hence, it can be assumed that the total size of the sets $$A(123),A(124),A(134),A(234),A(1234)$$
is at least 33, on the other hand, it is clearly at most 40. 
It follows that none of them is empty and at least three of them are 3-subspaces, so we can assume that $A(123)$ is a 3-subspace. $A(123)\leq \langle e_1,e_2,e_3,e_4\rangle$ is not possible, since then $e_1+e_2+e_3+A(123)$ would contain $e_1+e_2+e_4$ or $e_1+e_3+e_4$ or $e_2+e_3+e_4$ or $e_1+e_2+e_3+e_4$. 
Since, if an affine 3-subspace of $\langle e_1,e_2,e_3,e_4\rangle$ contains $e_1+e_2+e_3$ but none of the other 4 vectors, then it is $\langle e_1,e_2,e_3\rangle$, however, $\langle e_1,e_2,e_3\rangle$ contains 0, as well, contradiction.

So $e_1+e_2+e_3+A(123)$ intersects nontrivially $R$, so $|(e_1+e_2+e_3+A(123))\cap R|=4$, thus at least 4 subsets (among subsets $A(x)$ with $x\in R$)  are empty: $l\geq 4$. Note that that the sum of the four corresponding vectors is 0.
Also, note that in this case (similarly to \eqref{klelso} in the previous case) we shall assume that
\begin{equation}\label{klelso2}
3k\geq 5l.
\end{equation}

Now \eqref{klelso2} yields that at least 7 such subsets are 3-subspaces: $k\geq 7$. 
Then \eqref{klmasodik} implies that  the number of empty ones is at least 5. Again, by \eqref{klelso2} we get $k\geq 9$. If $l=5$, then we have $\binom{5}{3}=10$ triples, but there is a 4-term zero-sum, so 4 triples can not be ``empty triples'', thus there is a 6th empty subset: $l\geq 6$, and by \eqref{klelso2} we obtain that $k\geq 10$. So $\sum |A(z+e_5)|=10\cdot 8=80$. 
 As $125-5-80=40$, all the sets $A(123),A(124),A(134),A(234),A(1234)$ must be 3-subspaces. If $v\in \{e_1+e_2+e_3,e_1+e_2+e_4,e_1+e_3+e_4,e_2+e_3+e_4,e_1+e_2+e_3+e_4\}$, then $v+A(v)$ intersects $R$ in 4 vectors whose sum is 0. It can be checked that this set of 4 vectors can not be the same for all the 5 possible $v$-s. (Otherwise $\mathbb{F}_2^5\setminus R$ would contain at least 15 vectors to which the empty set is assigned, however, there are only 10 such vectors.) So there must be at least two such 4-element sets. Their intersection has size at least 2, since we have only 6 vectors in $R$ to which the empty set is assigned, and also at most 2, since otherwise they would be the same. Let $A(z_1+e_5),\dots,A(z_6+e_5)$ be the empty ones, and let us assume that the two 4-zero-sum-sets are $\{z_{1},\dots,z_4\}$ and $\{z_3,\dots,z_6\}$. Then $z_1+z_2=z_3+z_4=z_5+z_6$. 20 triples can be chosen out of these 6 vectors, but  just 8 of them can be ``empty triples'', contradiction.


Therefore, we can assume that there is no subset of type (A), that is, all the nonempty subsets are subspaces of dimension at most 3. According to Step 5 there must be at least one 3-subspace among the subsets, as Steps 6-7 imply that all the sets of size at least 5 are 3-subspaces. 


\smallskip
{\bf Step 8.} We show that the number of empty subsets is at least 13.
\smallskip

Let $1\leq k$ be the number of 3-subspaces and $l$ the number of empty subsets. Let us colour the elements of $\mathbb{F}_2^5$: $x$ is coloured red if $A(x)=\emptyset$ and $x$ is coloured blue if $A(x)$ is a 3-subspace. (If $A(x)$ is a subspace of dimension at most 2, then $x$ is not coloured.)
Let $\tilde{A}(x)=x+A(x)$, specially, if $x$ is blue, then $\tilde{A}(x)$ is a 3-dimensional affine subspace containing $x$ and seven red vectors.

If $125\leq S$, then $125\leq 8k+(32-k-l)4$ which yields $l\leq k$. 
Now we are going to show that $l\geq 13$.
If $x$ is blue, then in $\tilde{A}(x)$ there are two kinds of triples: the 2-subspace spanned by them either contains $x$ or not. The number of 
triples in $\tilde{A}(x)\setminus\{x\}$ is 35 and 7 of these triples span a 2-subspace containing $x$. These triples are not contained in any other affine 3-subspace $\tilde{A}(y)$.

Furthermore, we claim that if $l<13$, then a triple can appear in at most two 3-subspaces. For the sake of contradiction, let us assume that a triple is contained in $K\cap L\cap M$, where $K=\tilde{A}(x), L=\tilde{A}(y),M=\tilde{A}(z)$ are 3-subspaces. Let $H$ be the 2-subspace spanned by this triple, then $H=K\cap L\cap M$ and $K\setminus H$, $L\setminus H$, $M\setminus H$ are disjoint, thus $|K\cup L\cup M|=16$. However, in $K\cup L\cup M$ all the vectors are red except $x,y,z$, hence $16-3=13\leq l$, contradiction.

Now, since each triple appears in at most two 3-subspaces, we obtain that
$$7l+\frac{28l}{2}\leq 7k+\frac{28k}{2}\leq \binom{l}{3},$$
thus
$$126\leq (l-1)(l-2),$$
implying that $l\geq 13$.

Therefore, $k\geq l\geq 13$, as we claimed.


\smallskip
{\bf Step 9.} We show that if $A(x),A(y),A(z)$ are 3-subspaces (with distinct $x,y,z$), then $A(x)\cap A(y)\cap A(z)$ is not an affine 2-subspace. 
\smallskip

Now, for the sake of contradiction, assume that there are three 3-subspaces, $A(x),A(y),A(z)$ whose intersection is an affine 2-subspace $L$. Without loss of generality we can assume that $L$ is a linear (2-)subspace. Note that $\tilde{A}(x)=L\cup(L+x),\tilde{A}(y)=L\cup(L+y),\tilde{A}(z)=L\cup(L+z)$.

Note that $\mathbb{F}_2^5$ can be partitioned into 8 translates of $L$.
Every affine 3-subspace contains the same number of vectors from those $L$-translates that has a nonempty intersection with it. That is, given a 2-subspace $L$, we can distinguish three types of affine 3-subspaces, we are going to say that a 3-subspace is of
\begin{itemize}
\item type-1, if it contains 1-1 vector from each $L$-translate,
\item type-2, if it contains 2-2 vectors from four $L$-translates (and none from the remaining four $L$-translates),
\item type-4, if it contains 4-4 vectors from two $L$-translates (and none from the remaining six $L$-translates).
\end{itemize}
In $M=L\cup (L+x)\cup (L+y)\cup (L+z)$ there are 13 red elements, namely, all the vectors except $x,y,z$. If $t\notin M$  is blue, then $\tilde{A}(t)$ is a 3-subspace of type-1, type-2 or type-4 which contains $t$ and 7 seven red vectors. 

If at least two $L$-translates do not contain any red vector, then the elements of these translates can not be blue, so $k\leq 11$, contradiction. Hence, there is at most one $L$-translate without any red vector. In particular, this means that $l\geq 16$, since there are 13 red vectors in $M$ and at least 3 red vectors outside of $M$.

Thus $k=l=16$. Let us assume that the red vectors outside of $M$ are $v_1,v_2,v_3$, these vectors must be in different $L$-translates. Let $L'=\{u_1,u_2,u_3,u_4\}$ be the unique $L$-translate not containing any red vector. If $v_1+v_2+v_3\in L'$, then at most one of the $A(u_i)$ sets can be a 3-subspace (namely, $A(v_1+v_2+v_3)$), contradiction. Now assume that $v_1+v_2+v_3\notin L'$. By symmetry we can assume that $v_1+v_2+v_3\notin L+x$ also holds. But then the union of the $\tilde{A}(u_i)=\langle u_i,v_1,v_2,v_3\rangle_{aff}$ sets (that are all affine 3-subspaces of type-1) cover $L+x$ and the (unique) $u_i$ for which $x\in \tilde{A}(u_i)$ can not be blue (since $x$ is not red). Hence, no three-wise intersection of 3-subspaces can be a 2-subspace.


\smallskip
{\bf Step 10.} Now we know that $13\leq l\leq k$ and no three-wise intersection of 3-subspaces is a 2-subspace. We finish the proof of the upper bound 124 by verifying the statement in these cases.
\smallskip

Let $N$ be the number of those pairs of 3-subspaces whose intersection is a 2-subspace. Then
\begin{equation}
35k\leq \binom{l}{3}+4N,
\end{equation}
since each of the $k$ 3-subspaces contain 35 empty triples. Hence, for $l<16$ we have $N>0$, that is, two of the 3-subspaces assigned to blue vectors intersect each other in a 2-subspace. In the following subcases we always take two such subsets first.

{\bf Subcase 1.}
If $l=13$, then we can assume that $L$ is a linear 2-subspace and $\tilde{A}(x)=L\cup(L+x),\tilde{A}(y)=L\cup(L+y)$ are 3-subspaces corresponding to blue vectors $x$ and $y$. At least 2 translates of $L$ does not contain any red vector, and in these translates there can not be any blue vectors, either. So the number of blue vectors is at most $32-13-8=11$, contradiction.


{\bf Subcase 2.}
If $l=14$, then again let $L$ be a linear 2-subspace and $\tilde{A}(x)=L\cup(L+x),\tilde{A}(y)=L\cup(L+y)$ be 3-subspaces corresponding to blue vectors $x$ and $y$. Note that in $L\cup (L+x)\cup (L+y)$ there are 10 red vectors. We have 4 more red vectors, say, $v_1,v_2,v_3,v_4$, which must lie in different $L$-translates. (Otherwise there would be two $L$-translates without any red vector, which would imply that the 8 vectors in these translates are not coloured, contradicting that that the number of non-coloured vectors is at most 4.)

Note that all the 3-subspaces assigned to some blue vector different from $x,y$ are of type-1 or type-2. To get a 3-subspace of type-2 we need to take 2-2 red vectors from $L,L+x,L+y$. Moreover, these pairs must determine parallel vectors in these three $L$-translates (that is, in each pair the sum of the two vectors is the same), so there are at most 6 such subspaces. A type-1 3-subspace must correspond to a (blue) vector from the last $L$-translate, so there are at most 4 such subspaces. Hence $k\leq 4+6+2=12$, contradiction.


{\bf Subcase 3.}
Let us assume that $l=15$. Again, we can assume that for some linear 2-subspace $L$ the sets $A(x)=L\cup(L+x),A(y)=L\cup(L+y)$ are two 3-subspaces. Let $L_4,\dots,L_8$ be the remaining five $L$-translates. They contain altogether 5 red  vectors. If at least two of them do not contain any red vector, then in these two $L$-translates there aren't any blue vectors either, so the number of blue vectors is at most 9, contradiction. So without the loss of generality it can be assumed that either (i) $L_4$ contains two red vectors and $L_5,L_6,L_7$ contain one-one red vector: $v_i$ in $L_i$ ($5\leq i\leq 7$) or (ii) $L_4,\dots,L_8$ contain one-one red vector: $v_i$ in $L_i$ ($4\leq i\leq 8$).

In case (i) let $\alpha,\beta$ be the two directions that are different from the direction determined by the two red vectors of $L_4$. That is, $\alpha$ and $\beta$ are those two nonzero elements of $L$ that are different from the sum of the two red vectors in $L_4$.
Let us consider the following 6 vectors in $L_5,L_6,L_7$: $v_i+\alpha,v_i+\beta$ (for $5\leq i\leq 7$). If such a vector is blue, then the corresponding 3-subspace is of type-2, moreover, $L_1,L_2,L_3$ contain one-one red pair of this 3-subspace, and in each pair the sum is the same, either $\alpha$ or $\beta$. There are only 4 such triples (of pairs of vectors) meaning that at least two of the vectors $v_i+\alpha,v_i+\beta$ ($5\leq i\leq 7$) are not blue. To get 15 blue vectors all vectors in $L_8$ must be blue (as there are at most 2 non-coloured vectors). Note that the corresponding 3-subspaces must be of type-1. If $v_5+v_6+v_7\in L_8$, then there can be at most one blue element in $L_8$ (namely $v_5+v_6+v_7$). If $v_5+v_6+v_7\notin L_8$, then by symmetry we can also assume that $v_5+v_6+v_7\notin L_2$. If $t\in L_8$ is blue, then the corresponding 3-subspace is $\tilde{A}(t)=\langle v_5,v_6,v_7,t \rangle_{aff}$, but these four 3-subspaces cover $L_2$, which contradicts that $L_2$ contains only 3 red vectors.

In case (ii) there are two 3-subspaces of type-4. To get a 3-subspace of type-2, we have to choose one-one red pair from $L_1,L_2,L_3$ in such a way that these pairs determine parallel directions. This can be done in 6 ways, and every affine 3-subspace is determined by 6 points of it, so there are at most six 3-subspaces of type-2. To get a 3-subspace of type-1 we have to choose a red vector from all but one of the $L$-translates. First assume that no four-element subset of $\{v_4,\dots,v_8\}$ is a 2-subspace. Then the (at least) four red vectors chosen to be in this 3-subspace from  $\{v_4,\dots,v_8\}$ determine uniquely a 3-subspace, so the number of 3-subspaces of type-1 is at most 5, thus $k\leq 2+6+5=13$, a contradiction. Now assume that a 4-element subset, say, $\{v_4,v_5,v_6,v_7\}$ forms a 2-subspace. Each 3-subspace of type-1 contains at least 3 elements of $\{v_4,v_5,v_6,v_7\}$, hence all of them contain all these four vectors. Then the blue vector is in $L_8\setminus \{v_8\}$, so there are at most 3 such subspaces, thus, $k\leq 2+6+3=11$, a contradiction.


{\bf Subcase 4.}
Finally, let us assume that $l=k=16$, that is, all vectors are either red or blue. First we show that there are two 3-subspaces  whose intersection is a 2-subspace. For the sake of contradiction, assume the contrary. Let $S_1,S_2,S_3,S_4$ be four 3-subspaces assigned to blue vectors. If every pairwise intersection has size less than 4 (that is, the intersection is either empty or has size 2), then 
\begin{equation}\label{s1234}
|S_1\cup S_2\cup S_3\cup S_4|\geq \sum|S_i|-\sum|S_i\cap S_j|\geq 4\cdot 8-6\cdot 2=20,
\end{equation}
 so $S_1\cup S_2\cup S_3\cup S_4$ contains at least $20-4=16$ red vectors. Since there are only 16 red vectors, we must have equality in \eqref{s1234}, so each pairwise intersection has size 2 and each triple-intersection has size 0. Clearly, these hold for any four 3-subspaces assigned to blue vectors. Pick such a 3-subspace, for instance, $S_1$. Then the other fifteen 3-subspaces have to intersect $S_1$ in pairwise disjoint pairs, which is impossible. Therefore, there are two 3-subspaces whose intersection is a 2-subspace.


Hence, we can assume that this 2-subspace is a linear 2-subspace $L$ and the sets $A(x)=L\cup(L+x),A(y)=L\cup(L+y)$ are two 3-subspaces corresponding to blue vectors $x$ and $y$. Let $L_4,\dots,L_8$ be the remaining five $L$-translates. These contain 6 more red vectors. As there can be at most one $L$-translate without any red vector, we can assume that the number of red vectors among them is i) 3-1-1-1-0 or  ii) 2-2-1-1-0 or iii) 2-1-1-1-1.

In case (i) let $v_5,v_6,v_7$ be the red vectors in $L_5,L_6,L_7$. If $v_5+v_6+v_7\in L_8$, then in $L_8$ there is at most one blue vector (namely, $v_5+v_6+v_7$), contradiction. Assume that $v_5+v_6+v_7\notin L_8$. We can assume that $v_5+v_6+v_7\notin L_2$. If $t\in L_8$ is blue, then $\tilde{A}(t)=\langle t,v_5,v_6,v_7\rangle_{aff}$, but these cover $L_2$, which contradicts that $L_2$ contains a blue element.

In case (ii) let us assume that the direction $0\ne\alpha\in L$ is different from the direction(s) determined by the pairs in $L_4,L_5$. Let $v_6\in L_6,v_7\in L_7$ be the red vectors in these translates. Consider the blue vectors $v_6+\alpha$ and $v_7+\alpha$. The 3-subspaces corresponding to them are of type-2, and both of them contain one-one pair from $L_1,L_2,L_3$, moreover, all these pairs determine direction $\alpha$. In $L_2$ and $L_3$ these pairs are uniquely determined. In $L_1$ there are two choices (two disjoint pairs). However, these two pairs in $L_1$ together with the pairs from $L_2$ and $L_3$ determine two pairs in the same $L$-translate, which contradicts the existence of such a pair in both $L_6$ and $L_7$.

Finally, we consider case (iii).
Let $l_1=0,l_2=e_3,l_3=e_4,l_4=e_3+e_4,l_5=e_5,l_6=e_3+e_5,l_7=e_4+e_5,l_8=e_3+e_4+e_5$ and $L=\langle e_1,e_2\rangle$. For every $1\leq i\leq 8$ let $L_i=L+l_i$. We can assume that $L_1$ contains 4 red vectors and $L_2,L_3$ contains 3-3 red vectors.

First assume that the $L$-translate containing 2 red vectors is $L_4$, we can assume that these vectors are $e_3+e_4$ and $e_1+e_3+e_4$. Let $t$ be a blue vector in one of the four $L$-translates $L_5,\ldots,L_8$. Then $\tilde{A}(t)$ is either of type-2 or type-1. However, only $L_1,L_2,L_3,L_4$ contain at least two red vectors, which means that any 3-subspace of type-2 must contain at least 6 vectors from $L_1\cup L_2\cup L_3\cup L_4$, which is a 4-subspace, thus the remaining two vectors of the 3-subspace must also lie in this subspace, too. So $\tilde{A}(t)$ is of type-1. As $L_5\cup L_6\cup L_7\cup L_8$ is an affine 4-subspace, it intersects $\tilde{A}(t)$ in an affine 2-subspace. Therefore, if, say, $t\in L_8$, then $t$ and the red vectors from $L_5,L_6,L_7$ form an affine 2-subspace, that is, $t$ is the sum of these three red vectors. But then in $L_8$ the only blue vector is $t$, contradiction.


Hence, $L_4$ contains one red vector. By symmetry, we can assume that $L_5$ contains 2 red vectors and these are $e_5$ and $e_1+e_5$. Let the red vector in $L_i$ be $l_i+t_i$ for $i\in\{4,6,7,8\}$.

Let $i\in\{6,7,8\}$. We claim that $\tilde{A}(l_i+t_i+e_2)$ and $\tilde{A}(l_i+t_i+e_1+e_2)$ must be of type-1. Otherwise, $\tilde{A}(l_i+t_i+e_2)$ or $\tilde{A}(l_i+t_i+e_1+e_2)$ would contain at least two vectors from $L_5\cup L_6\cup L_7\cup L_8$, so it would have to contain two more red vectors from one of the $L$-translates $L_5,L_6,L_7,L_8$, these could only be $e_5$ and $e_5+e_1$ from $L_5$. But then the blue vector $l_i+t_i+e_1$ would also lie in the 3-subspace (to get parallel pairs from the different translates), a contradiction.
Hence, $\tilde{A}(l_i+t_i+e_2)$ and $\tilde{A}(l_i+t_i+e_1+e_2)$ are of type-1. 


Consider  $\tilde{A}(l_6+t_6+e_2)$ and $\tilde{A}(l_6+t_6+e_1+e_2)$.
Each of these  two subspaces contain either $e_5$ or $e_5+e_1$ and they contain $l_7+t_7$ and $l_8+t_8$. As the intersection of $\tilde{A}(l_6+t_6+e_2)\cap \tilde{A}(l_6+t_6+e_1+e_2)$ with the 1-codimensional affine subspace $L_5\cup L_6 \cup L_7\cup L_8$ must be of size 2, they contain different elements from $L_5$. Without loss of generality we can assume that $\tilde{A}(l_6+t_6+e_2)$ contains $e_5$. Then $e_5+(l_6+t_6+e_2)+(l_7+t_7)+(l_8+t_8)=0$, that is, $t_6+t_7+t_8=e_2$. 
Now $\tilde{A}(l_6+t_6+e_2)$ and $\tilde{A}(l_6+t_6+e_2+e_1)$ are determined, since they must contain $l_4+t_4$:
\begin{multline*}\tilde{A}(l_6+t_6+e_2)=
\{l_1+t_4+t_8,l_2+t_4+t_6+t_8+e_2,l_3+t_4+t_7+t_8,l_4+t_4,l_5,\\
l_6+t_6+e_2,l_7+t_7,l_8+t_8\},
\end{multline*}
\begin{multline*}
\tilde{A}(l_6+t_6+e_2+e_1)=\{l_1+t_4+t_8+e_1,l_2+t_4+t_6+t_8+e_2+e_1,l_3+t_4+t_7+t_8,l_4+t_4,\\
l_5+e_1,l_6+t_6+e_2+e_1,l_7+t_7,l_8+t_8\}.
\end{multline*}
Similarly, the type-1 3-subspaces containing 2-2 blue vectors from $L_7$ and $L_8$ are: 
\begin{multline*}
\{l_1+t_4+t_8,l_2+t_4+t_6+t_8,l_3+t_4+t_7+t_8+e_2,l_4+t_4,\\
l_5,l_6+t_6,l_7+t_7+e_2,l_8+t_8\},\\
\{l_1+t_4+t_8,l_2+t_4+t_6+t_8,l_3+t_4+t_7+t_8+e_2+e_1,l_4+t_4,\\
l_5+e_1,l_6+t_6,l_7+t_7+e_2+e_1,l_8+t_8\},\\
\{l_1+t_4+t_8+e_2,l_2+t_4+t_6+t_8+e_2,l_3+t_4+t_7+t_8+e_2,l_4+t_4,\\
l_5,l_6+t_6,l_7+t_7,l_8+t_8+e_2\},\\
\{l_1+t_4+t_8+e_2+e_1,l_2+t_4+t_6+t_8+e_2+e_1,l_3+t_4+t_7+t_8+e_2+e_1,l_4+t_4,\\
l_5+e_1,l_6+t_6,l_7+t_7,l_8+t_8+e_2+e_1\}.
\end{multline*}
So the set of red vectors in $L_2$ is $\{l_2+t_4+t_6+t_8,l_2+t_4+t_6+t_8+e_2,l_2+t_4+t_6+t_8+e_2+e_1\}$ and in $L_3$ is $\{l_3+t_4+t_7+t_8,l_3+t_4+t_7+t_8+e_2,l_3+t_4+t_7+t_8+e_2+e_1\}$.

Now consider $l_8+t_8+e_1$ which is a blue vector in $L_8$. Note that $\tilde{A}(l_8+t_8+e_1)$ is of type-2 (otherwise the three 3-subspaces corresponding to blue vectors from $L_8$ would have a 2-subspace intersection, contradicting Step 9). Also, it must contain $l_8+t_8$. As it contains at least 2 vectors from the 1-codimensional affine subspace $L_5\cup L_6\cup L_7 \cup L_8$, it must contain two more, which can only be $l_5,l_5+e_1$. The remaining two red pairs are in two of $L_1,L_2,L_3$. As $L_8=L_5+(e_3+e_4)$, these two $L$-translates must be $L_2$ and $L_3$. Also, the difference of the vectors from the same $L$-translate must be $e_1$, so the 3-subspace is:
\begin{multline*}
\{l_2+t_4+t_6+t_8+e_2,l_2+t_4+t_6+t_8+e_2+e_1,l_3+t_4+t_7+t_8+e_2,l_3+t_4+t_7+t_8+e_2+e_1,\\l_5,l_5+e_1,l_8+t_8,l_8+t_8+e_1                 \}  .
\end{multline*}

As $\{l_2+t_4+t_6+t_8+e_2,l_2+t_4+t_6+t_8+e_2+e_1\}=\{l_5,l_5+e_1\}+e_5+e_3+t_4+t_6+t_8+e_2$, we get that
$\{l_8+t_8,l_8+t_8+e_1\}+e_5+e_3+t_4+t_6+t_8+e_2=\{l_3+t_4+t_6+e_2,l_3+t_4+t_6+e_2+e_1\}$ has to coincide with $\{l_3+t_4+t_7+t_8+e_2, l_3+t_4+t_7+t_8+e_2+e_1  \}$. However, this leads to $t_6+t_7+t_8\in \{0,e_1\}$, contradiction.

\section{4AP-free subsets of $\mathbb{Z}_4^n$}
\label{sec4ap}

\begin{proof}[Proof of Theorem \ref{thm:r_4(Z_4^n}]
According to Lemma~\ref{4apeq} it suffices to show that $r_4'(1)=3, r_4'(2)=10, r_4'(3)=36$ and $r_4'(4)=128$.
In other words, we will show that if the system of subsets  $\{ A(x)\subseteq \mathbb{F}_2^n\ |\ x\in\mathbb{F}_2^n \}$ satisfies $(**)$, then $S=\sum\limits_{x\in\mathbb{F}_2^n} |A(x)|$ is at most 3, 10, 36, 128 for $n=1,2,3,4$, respectively. Then, we will present constructions of these sizes.

By the pigeon-hole principle we get that $A(x)+A(x)=\mathbb{F}_2^n$ if $|A(x)|> 2^{n-1}$. Hence, $|A(x)|> 2^{n-1}$ holds for at most one $x$, since $x\ne y$ and $2^{n-1}< |A(x)|,|A(y)|$ would imply that $x+y\in (A(x)+A(x))\cap(A(y)+A(y))=\mathbb{F}_2^n$, contradicting $(**)$.

This observation immediately yields that $S\leq 2^n+(2^n-1)2^{n-1}=2^{2n-1}+2^{n-1}$. For $n=1,2,3$ we obtain the claimed upper bounds $3,10,36$, respectively.

For $n=4$ we obtain that $S\leq 136$, now we will show that $S\leq 128$ also holds. We have already seen (in the proof of Theorem~\ref{3APthm}) that it can be assumed that all the nonempty $A(x)$ subsets are linear subspaces or a set of 5 affine independent points. If all the subsets are of size at most 8, then clearly $S\leq 16\cdot 8=128$. So we can assume that one of them is $\mathbb{F}_2^4$, without loss of generality let $A(0)=\mathbb{F}_2^4$. It can be assumed that the number of 3-subspaces among the $A(x)$ sets is at least 13, since otherwise $S\leq 16+12\cdot 8+3\cdot 5=127$. If $A(x)$ is a 3-subspace, then for some (uniquely determined) $\varphi(x)\in\mathbb{F}_2^4$ we have $A(x)=(\varphi(x))^\perp$. As $x+0\notin (A(x)+A(x))\cap (A(0)+A(0))=A(x)$, we obtain that $x\varphi(x)=1$. We claim that $\varphi$ is injective, that is, if $A(x)$ and $A(y)$ are 3-subspaces (with $x\ne y$), then $\varphi(x)\ne \varphi (y)$. Otherwise, $(x+y)\varphi(x)=x\varphi(x)+y\varphi(y)=1+1=0$, so $x+y\in A(x)$ and similarly $x+y\in A(y)$. So this would lead to $x+y\in (A(x)+A(x))\cap (A(y)+A(y))$, which contradicts property $(**)$. Therefore, $\varphi$ is  injective. Also, if $A(x)$ and $A(y)$ are 3-subspaces (and $x\ne y$), then $x\varphi(y)=0$ or $y\varphi(x)=0$, since $x\varphi(y)=y\varphi(x)=1$ would imply that $(x+y)\varphi(x)=0=(x+y)\varphi(y)$ and so $x+y\in (A(x)+A(x))\cap (A(y)+A(y))$, which would contradict property $(**)$.

Now, let us assume that for some $z$ the set $A(z)$ is a set of 5 affine independent points. Let $X=\{x\in\mathbb{F}_2^4\ |\ x+z\in A(z)\hat{+}A(z)\}$. As $z=z+0\notin (A(z)\hat{+}A(z))\cap(A(0)\hat{+}A(0))=A(z)\hat{+}A(z)$, we have $X\subseteq \mathbb{F}_2^4\setminus \{0,z\}$. 
Note that $A(z)\hat{+}A(z)$ contains $\binom{5}{2}$ (distinct) sums, thus we have $|X|=10$. For all $x\in X$ we have $x+z\notin A(x)\hat{+}A(x)$. We know that at least 13 subsets are 3-subspaces, so there are at most three subsets that are not 3-subspaces: $A(0),A(z)$ and possibly one more. Thus for at least 9 elements of $X$ the set $A(x)$ is a 3-subspace. For such an $x$ the condition $x+z\notin A(x)\hat{+}A(x)$ implies that $1=(x+z)\varphi(x)=1+z\varphi(x)$, hence $z\varphi(x)=0$. As $\varphi$ is injective, this would mean that the 3-subspace $(z)^\perp$ contains at least 9 different vectors, which is a contradiction. Hence, it can be assumed that all the nonempty $A(x)$ sets are linear subspaces.

At least 14 of the $A(x)$ subsets are 3-subspaces, since otherwise $S\leq 16+13\cdot 8+2\cdot 4=128$ clearly holds, as $|A(x)|\leq 4$ for every $x\ne 0$ for which $A(x)$ is not a 3-subspace. Therefore, the mapping $\varphi$ is defined on $\mathbb{F}_2^4\setminus\{0\}$ with the exception of at most one point. Also, $\varphi$ is injective, so it can be extended to a bijective mapping from $\mathbb{F}_2^4\setminus\{0\}$ to $\mathbb{F}_2^4\setminus\{0\}$. 
Let $H=\{x: A(x)\text{ is a 3-subspace}\}$. Then either $H=\mathbb{F}_2^4\setminus \{0\}$ or $H=\mathbb{F}_2^4\setminus \{0,u\}$ for some $u$. Let
$$N:=|\{(x,y): x,y\in H,x\ne y,x\varphi(y)=0\}|.$$

At first assume that $H=\mathbb{F}_2^4\setminus \{0\}$. 
As at least one of $x\varphi(y)$ and $y\varphi(x)$ is equal to 0 for every $x\ne y$, we get that $N\geq \binom{15}{2}=105$. On the other hand $N\leq \sum\limits_{x\in H} (|x^\perp|-1)=15\cdot 7=105$. Therefore, $|N|=105$ and for any two distinct elements of $H$ exactly one of $x\varphi(y)$ and $y\varphi(x)$ is equal to 0. In other words, $x\varphi(y)+y\varphi(x)=1$ for any two different elements $x,y\in H$. 
Let $u(x)=(1,x,\varphi(x)),v(x)=(1,\varphi(x),x)\in\mathbb{F}_2^9$ for every $x\in H$. Then $u(x)v(y)=\delta_{xy}$, thus $\{u(x),v(x)\}_{x\in H}$ is a biorthogonal system, implying that $|H|\leq \dim \mathbb{F}_2^9= 9$, which is a contradiction.
  
Now assume that there is a subset $A(u)$ (with $u\ne 0$) which is not a 3-subspace: $H=\mathbb{F}_2^4\setminus \{0,u\}$. 
As at least one of $x\varphi(y)$ and $y\varphi(x)$ is equal to 0 for every $x\ne y$, we get that $N\geq \binom{14}{2}=91$. However, $N\leq \left(\sum\limits_{x\in H} (|x^\perp|-1) \right)-|u^\perp\cap H|\leq 14\cdot 7-6=92$. Hence, $N\in\{91,92\}$ and there is at most one pair of distinct elements $x,y\in H$ such that $x\varphi(y)=y\varphi(x)=0$.
By dropping out one of the two elements of this pair from $H$ (if such a pair exists at all) we obtain a 13-element subset $H'\subseteq H$ such that $x\varphi(y)+y\varphi(x)=1$ for every $x,y\in H'$, $x\ne y$.
Again, let $u(x)=(1,x,\varphi(x)),v(x)=(1,\varphi(x),x)\in\mathbb{F}_2^9$ for every $x\in H'$. Then $u(x)v(y)=\delta_{xy}$, thus $\{u(x),v(x)\}_{x\in H'}$ is a biorthogonal system, implying that $|H'|\leq 9$, which is a contradiction.

Hence, it is shown that $S\leq 128$.

Now we give constructions to prove the lower bounds.

\noindent
{\bf Case 1: $n=1$.}

$A(0)=\mathbb{F}_2,A(1)=\{0\}$ give $3\leq r_4'(1)$. (In fact, any 3-element subset of $\mathbb{Z}_4$ is free of arithmetic progressions of length 4, trivially.)

\noindent
{\bf Case 2: $n=2$.}

Let $A(0)=\mathbb{F}_2^2=\langle e_1,e_2\rangle$. Furthermore, let $\varphi(e_1)=e_1, \varphi(e_2)=e_1+e_2, \varphi(e_1+e_2)=e_2$. Then $x\varphi(x)=1$ for every $x\ne 0$ and $x\varphi(y)+y\varphi(x)=1$ for every $x,y\in \mathbb{F}_2^2\setminus\{0\}$, $x\ne y$. For $0\ne x$ let $A(x)=(\varphi(x))^\perp$.
Then $x+0\notin A(x)$, since $x\varphi(x)=1$. Also, for any two nonzero vectors $x$ and $y$ either 
$x\varphi(y)=0$ or $y\varphi(x)=0$. We can assume that $x\varphi(y)=0$. (Otherwise we swap $x$ and $y$.) Then $(x+y)\varphi(y)=0+1$ implies that $x+y\notin A(y)=A(y)+A(y)$, so the condition $(**)$ holds. Thus $10\leq r_4'(2)$.

\noindent
{\bf Case 3: $n=3$.}
Let $A(0)=\mathbb{F}_2^3=\langle e_1,e_2,e_3\rangle$. Similarly to the previous case it suffices to define a bijective mapping $\varphi:\mathbb{F}_2^3\setminus\{0\} \to \mathbb{F}_2^3\setminus\{0\}$ such that $x\varphi(x)=1$ for every $x\ne 0$ and $x\varphi(y)+y\varphi(x)=1$ for every $x\ne y$. It is easy to check that the following mapping satisfies these conditions:
$\varphi (e_1)=e_1, \varphi(e_2)=e_1+e_2,\varphi(e_3)=e_1+e_2+e_3,\varphi(e_1+e_2)=e_2+e_3,\varphi(e_1+e_3)=e_3, \varphi(e_2+e_3)=e_1+e_3,\varphi(e_1+e_2+e_3)=e_2$.
Hence, $8+7\cdot 4=36\leq r_4'(3)$.

\noindent
{\bf Case 4: $n=4$.}

Let $\mathbb{F}_2^4=\langle e_1,e_2,e_3,e_4 \rangle$. Let us extend the mapping $\varphi:\langle e_1,e_2,e_3\rangle \to \langle e_1,e_2,e_3\rangle $ defined in Case 3 with $\varphi(0)=0$. For every $x\in \langle e_1,e_2,e_3\rangle $ let $A(x)=A(x+e_4)=(\varphi(x)+e_4)^\perp$. Let $x,y\in \langle e_1,e_2,e_3 \rangle$ and $\alpha,\beta\in\{0,1\}$. We have to show that 
$$(x+\alpha e_4)+(y+\beta e_4)\notin A(x+\alpha e_4)\cap A(y+\beta e_4)$$
unless $x=y$ and $\alpha=\beta$. If $(x+\alpha e_4)+(y+\beta e_4)\in A(x+\alpha e_4)$, then $(x+y+(\alpha+\beta)e_4)(\varphi(x)+e_4)=0$, that is, $x\varphi(x)+y\varphi(x)+\alpha+\beta=0$. Similarly, $(x+\alpha e_4)+(y+\beta e_4)\in A(y+\beta e_4)$ implies that $y\varphi(y)+x\varphi(y)+\alpha+\beta=0$.

If $x=y$, then $0=x\varphi(x)+x\varphi(x)+\alpha+\beta$ yields $\alpha=\beta$, and we are done. From now on, let us assume that $x\ne y$.

If $x=0$, then by adding up the two equations: $0=0\varphi(0)+y\varphi(0)+y\varphi(y)+0\varphi(y)=y\varphi(y)=1$, which is a contradiction. Similarly, $y=0$ also leads to a contradiction.

Finally, let us assume that $x\ne y$ and $x,y\ne 0$. Then by adding up the two equations we get $0=x\varphi(x)+y\varphi(y)+(x\varphi(y)+y\varphi(x))=1+1+1=1$, which is a contradiction, too.

Hence, the system satisfies property $(**)$, and $16\cdot 8\leq r_4'(4)$.

\end{proof}

\section{Acknowledgements}
C.E. was partially supported by FWF grant W1230, 
P.P.P. was supported by the National Research, Development and Innovation Office of Hungary
  (Grant Nr.  PD115978 and  K129335) and the J\'anos Bolyai Research
  Scholarship of the Hungarian Academy of Sciences.

\end{document}